\documentclass[oneside, 11pt]{amsart}

\usepackage{amsmath}
\usepackage{amsthm}
\usepackage{amssymb}
\usepackage[foot]{amsaddr}
\usepackage[utf8]{inputenc}
\usepackage{comment}
\usepackage{array}
\usepackage{tikz}
\usepackage{hyperref}
\usepackage{a4}
\usepackage{enumitem}
\usepackage{microtype}

\DeclareMathOperator{\Def}{Def}

\DeclareMathOperator{\Hom}{Hom}
\DeclareMathOperator{\id}{id}
\DeclareMathOperator{\kt}{KT}
\DeclareMathOperator{\mon}{Mon}
\DeclareMathOperator{\ns}{NS}

\DeclareMathOperator{\rk}{rk}
\DeclareMathOperator{\sig}{sign}

\newcommand{\C}{\mathbb{C}}
\newcommand{\Q}{\mathbb{Q}}
\newcommand{\K}{\mathbb{K}}
\newcommand{\R}{\mathbb{R}}
\newcommand{\Z}{\mathbb{Z}}
\newcommand{\Pb}{\mathbb{P}}

\newcommand{\frakM}{\mathfrak{M}}

\newcommand{\calC}{\mathcal{C}}
\newcommand{\calD}{\mathcal{D}}

\newcommand{\calK}{\mathcal{K}}
\newcommand{\calL}{\mathcal{L}}
\newcommand{\calM}{\mathcal{M}}
\newcommand{\calO}{\mathcal{O}}

\newcommand{\calT}{\mathcal{T}}

\newcommand{\calX}{\mathcal{X}}

\newcommand{\Mp}{{M^\perp}}
\newcommand{\gmp}{\Gamma_{\Mp}}

\newcommand{\kn}{$K3^{[n]}$}
\newcommand{\cx}{\calC_X}
\newcommand{\kx}{\calK_X}

\newcommand{\mk}{\frakM_{\Mp,\,\calK}}
\newcommand{\ml}{\frakM_{L_n}^0}
\newcommand{\monh}{\mon^2_{\textup{Hdg}}}
\newcommand{\sik}{\widetilde{\calK}^i_X}
\newcommand{\Ln}{L_{n}}
\newcommand{\Dp}{L_n(M)}
\newcommand{\Dl}{\Delta_M(L_n)}
\newcommand{\pk}{P_\calK}
\newcommand{\dd}{\delta}
\newcommand{\so}{\widetilde{O}}
\newcommand{\tx}{\calT_X}

\newcommand\isoarrow{\xrightarrow{\,\smash{\raisebox{-0.65ex}{\ensuremath{\sim}}}\,}}

\newtheorem{thm}{Theorem}[section]
\newtheorem{prop}[thm]{Proposition}
\newtheorem{lem}[thm]{Lemma}

\theoremstyle{definition}

\newtheorem{dfn}[thm]{Definition}
\newtheorem{rem}[thm]{Remark}
\newtheorem{exm}[thm]{Example}

\numberwithin{equation}{section}

\title[]{Non-symplectic involutions of irreducible symplectic manifolds of $K3^{[n]}$-type}
\date{\today}
\author{Malek Joumaah}
\address{Institut f\"ur Algebraische Geometrie, Leibniz Universit\"at Hannover, Welfengarten 1, 30060 Hannover, Germany}
\email{joumaah@math.uni-hannover.de}

\begin{document}

\begin{abstract}
This paper is concerned with non-symplectic involutions of irreducible symplectic manifolds of \kn-type. We will give a criterion for deformation equivalence and use this to give a lattice-theoretic description of all deformation types.
While moduli spaces of \kn-type manifolds with non-symplectic involutions are not necessarily Hausdorff,
we will construct quasi-projective moduli spaces for a certain well-behaved class of such pairs.
\end{abstract}


\maketitle


\section{Introduction}

\thispagestyle{empty}  

In the last 35 years the study of automorphisms of K3 surfaces has attracted much attention.
Important results on the classification of such automorphisms have been obtained in particular by Nikulin, Kondo, and Mukai. 
First of all, any finite automorphism group $G$ of a K3 surface $S$ is the extension of a cyclic group by the subgroup of $G$ acting symplectically on $S$. In the abelian case, finite symplectic automorphism groups of K3 surfaces have first been classified by Nikulin \cite{nikulin3}. Later, Mukai \cite{mukai} showed that finite groups acting symplectically on some K3 surface coincide with those groups admitting a certain type of embedding into the Mathieu group $M_{23}$. An independent classification was later given by Xiao \cite{xiao}.

Non-symplectic automorphisms have been studied by Nikulin \cite{nikulin2}, Kondo \cite{kondononsympl}, Oguiso--Zhang \cite{oguisozhang}, Artebani--Sarti--Taki \cite{artebanisartitaki}, and others.
For non-symplectic involutions $i:S\to S$, two main results are the following:

\begin{enumerate}
\item The deformation type of a pair $(S,i)$ is determined by the invariant sublattice of the induced isometry $i^*:H^2(S,\Z)\to H^2(S,\Z)$. \cite{nikulin2}
\item For a given isometry class of an invariant sublattice $M\subset L_{K3}$, the corresponding moduli space is a Zariski-open subset of an arithmetic quotient $\Omega_{M^\perp}/\Gamma_\Mp$ of a bounded symmetric domain, and in particular a quasi-projective variety. \cite{yoshikawa}
\end{enumerate}

\smallskip

Irreducible symplectic manifolds are higher-dimensional analogues of K3 surfaces, and their automorphisms have been studied in recent years by several authors.  

Boissi\`ere--Nieper-Wißkirchen--Sarti \cite{boissiereniepersarti} and Oguiso--Schröer \cite{oguisoschroer} gave examples of non-symplectic automorphisms without fixed points, and hence of higher-dimensional analogues of Enriques surfaces.
Classification results have been obtained in particular for manifolds of \kn-type, which are $2n$-dimensional irreducible symplectic manifolds deformation equivalent to the  Hilbert scheme of $n$ points on a K3 surface.
Hoehn and Mason \cite{mason} generalized Mukai's results on finite symplectic automorphism groups to the $K3^{[2]}$-case, and
Mongardi \cite{mongardisymplectic} classified the deformation types of $K3^{[2]}$-type manifolds with symplectic prime order automorphisms. 
A classification of the invariant lattices of non-symplectic prime order automorphisms in the $K3^{[2]}$-case has been given by Boissi\`ere, Camere, and Sarti \cite{boissierecameresarti}.
Beauville \cite{beauvilleinv} started the systematic study of non-symplectic involutions of \kn-type manifolds and gave a description of their fixed loci. In the case $n=2$, Mongardi and Wandel \cite{mongardiwandel} constructed examples realizing almost every possible invariant lattice.

\smallskip

 In this paper we will in particular consider the moduli problem for irreducible symplectic manifolds of \kn-type with non-symplectic involutions.
As in the K3 case, the key tool will be the Global Torelli theorem, which has recently been generalized by Verbitsky to irreducible symplectic manifolds, see \cite{verbitskytorelli}, \cite{markmansurvey}.
We will see, however, that in general neither of the statements (i) and (ii) above remains true for \kn-type manifolds. The reason is that, unlike K3 surfaces, higher-dimensional irreducible symplectic manifolds
can possess several birational models.

Markman \cite[\S 5]{markmansurvey} showed that the different birational models of an irreducible symplectic manifold $X$ correspond to a chamber decomposition of the positive cone of $X$, which by recent results of Amerik and Verbitsky \cite{amerikverbitskyrational} is defined by hyperplanes orthogonal to a certain deformation invariant set of divisor classes $\Delta(X)\subset H^{1,1}(X,\Z)$ (the definition will be given in Subsection \ref{subsec:ktchambers}).

If $i:X\to X$ is a non-symplectic involution and $f:X\dashrightarrow\widetilde{X}$ is a birational model, then it can happen that the induced involution $\tilde{i}=f\circ i\circ f^{-1}:\widetilde{X}\dashrightarrow\widetilde{X}$ is again biregular (see Example \ref{exm:comp}).
If the chambers corresponding to $X$ and $\widetilde{X}$ are separated by a wall $D^\perp$ for some $D\in\Delta(X)$, there are two possibilities:
\begin{enumerate}[label=(\alph*)]
\item The divisor $D$ is invariant under $i$. In this case, the wall $D^\perp$ is stable under deformation of $(X,i)$ and the two pairs $(X,i)$ and $(\widetilde{X},\tilde{i})$ deform (locally) into different families (see Proposition \ref{prop:definv}).
\item The divisor $D$ is not invariant under $i$. In this case, the wall $D^\perp$ vanishes for a generic deformation of $(X,i)$, and therefore $(X,i)$ and $(\widetilde{X},\tilde{i})$ deform into the same family (Proposition \ref{prop:insep}).
\end{enumerate}
As a consequence of possibility (a), there can be more than one deformation type of involutions, even if the invariant sublattice is fixed up to parallel transport, rather than up to isometry. In Section \ref{sec:sikc} we will associate another invariant to a non-symplectic involution $i:X\to X$, the stable invariant Kähler cone $\widetilde{\calK}_X^i$. This is a cone containing the invariant Kähler cone $\calK_X^i$ and consists of classes that deform into an invariant Kähler class for the generic deformation of $(X,i)$. In Section \ref{sec:defeq} we will show that two \kn-type manifolds with non-symplectic involutions are deformation equivalent if and only if their stable invariant Kähler cones are equivalent under parallel transport (Proposition \ref{prop:definv} and Theorem \ref{thm:def}). 

Moreover, starting with an admissible sublattice $M\subset L_n$ of the abstract \kn-lattice (see Definition \ref{dfn:admissiblelattice}), we will give a lattice-theoretic description of the deformation types of involutions of type $M$ (Theorem \ref{thm:deftypes}). In particular, this will show that every admissible sublattice is isometric to an invariant sublattice of a non-symplectic involution.

In Section \ref{sec:mod} we will study moduli spaces of \kn-type manifolds with non-symplectic involutions. As a consequence of possibility (b) above, a Hausdorff (and in particular a quasi-projective) moduli space does not always exist. 
We will therefore restrict to \emph{simple} pairs $(X,i)$, that is, those satisfying $\widetilde{\calK}_X^i=\calK_X^i$. 
We show, that non-simple pairs correspond to a divisor in the local deformation space (Proposition \ref{prop:simple}) and that a quasi-projective moduli space for simple pairs exists (Theorem \ref{thm:modnat}).


\section{Lattices}

A \emph{lattice} is a finitely generated free abelian group $L$ together with a non-degenerate symmetric bilinear form $(\cdot,\cdot):L\times L\to \Z$. The rank of $L$ is denoted by $\rk(L)$.

The bilinear form defines an embedding $L\hookrightarrow L^*$ of $L$ into the \emph{dual lattice} $L^*:=\Hom(L,\Z)$ as a finite index subgroup. The \emph{discriminant group} of $L$ is the finite abelian group $A_L:=L^*/L$.
A lattice is called \emph{unimodular}, if $L=L^*$, and \emph{even}, if $(l,l)\in 2\Z$ for every $l\in L$. We denote by $U$ the even unimodular lattice of signature $(1,1)$, and by $E_8(-1)$ the even unimodular lattice of signature $(0,8)$. The rank one lattice generated by an element $l$ such that $(l,l)=k$ is denoted by $\langle k\rangle$.

For any field $\K\in\{\Q,\R,\C\}$ we consider the vector space $L_\K:=L\otimes_\Z\K$ together with the induced $\K$-valued bilinear form.  The isomorphism $L_\Q\cong L_\Q^*$ induces a $\Q$-valued bilinear form on the dual lattice $L^*\subset L_\Q^* \cong L_\Q$. For any even lattice $L$, this defines a $\Q/2\Z$-valued quadratic form on the discriminant group $A_L$.

We denote by $O(L)$ and $O(A_L)$ the isometry groups of $L$ and $A_L$, respectively. An isometry $\varphi\in O(L)$ of a lattice induces an isometry $\overline{\varphi}\in O(A_L)$ of its discriminant group. The \emph{stable isometry group} $\widetilde{O}(L)$ is the kernel of the natural homomorphism $O(L)\to O(A_L)$. Since $A_L$ is finite, $\widetilde{O}(L)\subset O(L)$ is a finite index subgroup.

\begin{lem} \label{lem:stableisometry}
Let $S\subset L$ be a sublattice and $\varphi\in\widetilde{O}(S)$. Then $\varphi$ extends to an isometry $\widetilde{\varphi}\in\widetilde{O}(L)$ such that $\widetilde{\varphi}|_{S^\perp}=\id_{S^\perp}$.
\end{lem}
\begin{proof}
\cite[Lemma 7.1]{ghs2}
\end{proof}

\begin{lem} \label{lem:orbits}
Let $L$ be a lattice and $k\in \Z$. There are only finitely many $O(L)$-orbits of elements $l\in L$ with $(l,l)=k$.
\end{lem}
\begin{proof}
\cite[Satz 30.2]{kneser}
\end{proof}

A sublattice $S\subset L$ is \emph{primitive}, if the quotient group $L/S$ is free. Two primitive sublattices $S\subset L$ and $S'\subset L'$ are \emph{isometric}, if there exists an isometry $\varphi:L\to L'$ with $\varphi(S)=S'$.

Let $S\subset L$ be a primitive sublattice and $K:=S^\perp\subset L$ its orthogonal complement. The sequence of inclusions
$S\oplus K\subset L\subset L^*\subset S^*\oplus K^*$
defines an inclusion $H_L:=L/(S\oplus K)\subset A_S\oplus A_K$ as an isotropic subgroup.
 The restricted projections 
$p_S:H_L\to H_S:=p_S(H_L)$
and
$p_K:H_L\to H_K:=p_K(H_L)$ are isomorphisms of groups, and the isomorphism 
\[\gamma:=p_K\circ p_S^{-1}:H_S\to H_K\] is an anti-isometry.

Now consider another primitive sublattice $S'\subset L$ with orthogonal complement $K'$ and let $\gamma':H_{S'}\to H_{K'}$ be as above.

\begin{prop} \label{prop:embeddings}
Let $\varphi:S\to S'$ and $\psi:K\to K'$ be isometries. The isometry 
\[\varphi\oplus\psi:S\oplus K\to S'\oplus K'\]
extends to an isometry of $L$ if and only if $\overline{\psi}\circ\gamma=\gamma'\circ\overline{\varphi}$.
\end{prop}
\begin{proof}
\cite[Cor. 1.5.2]{nikulin1}
\end{proof}


\section{Irreducible symplectic manifolds}

\begin{dfn}
An \emph{irreducible (holomorphic) symplectic manifold} is a complex manifold $X$, such that
\begin{enumerate} \setlength\itemsep{2.5pt}
\item $X$ is a compact Kähler manifold,
\item $X$ is simply connected, 
\item $H^0(X,\Omega_X^2)=\C\omega$, where $\omega$ is an everywhere non-degenerate holomorphic 2-form on $X$.
\end{enumerate}
\end{dfn}

The non-degeneracy of $\omega$ implies that the complex dimension of $X$ is even. In dimension 2, irreducible symplectic manifolds coincide with K3 surfaces. For a K3 surface $S$ and an integer $n\geq 2$, we denote by $S^{[n]}$ the Hilbert scheme of zero-dimensional length $n$ subschemes of $S$. Beauville \cite{beauvilleihs} showed that $S^{[n]}$ is an irreducible symplectic manifold of dimension $2n$. (For $n=2$, this was first shown by Fujiki \cite{fujikisquare}.) An irreducible symplectic manifold is called of \emph{$K3^{[n]}$-type}, if it is deformation equivalent to $S^{[n]}$ for a $K3$ surface $S$.

\begin{prop}
For an irreducible symplectic manifold $X$, the group $H^2(X,\Z)$ carries a natural integral symmetric bilinear form $(\cdot,\cdot)$ of signature $(3,b_2(X)-3)$ satisfying
\begin{enumerate} \setlength\itemsep{2.5pt}
\item $(\omega,\omega)=0,\quad(\omega,\overline{\omega})>0$,
\item $H^{1,1}(X)=(H^{2,0}(X)\oplus H^{0,2}(X))^\perp\subset H^2(X,\C)$,
\item $(x,x)>0$ for every Kähler class $x$ on $X$.
\end{enumerate}
\end{prop}
\begin{proof}
\cite{beauvilleihs}
\end{proof}

This form is called the \emph{Beauville--Bogomolov form} (or sometimes Beauville--Bogomolov--Fujiki form) of $X$.

Note that property (ii) implies that
\begin{equation} \label{eq:ns}
\ns(X)=H^{1,1}(X,\Z):=H^2(X,\Z)\cap H^{1,1}(X)=H^2(X,\Z)\cap\{\omega^\perp\}.
\end{equation}

\begin{exm}
For a K3 surface $S$, the Beauville--Bogomolov form coincides with the intersection form, that is, 
\[H^2(S,\Z)\cong L_{K3}:=3U\oplus 2E_8(-1).\] 
By a result of Beauville \cite[Prop. 6]{beauvilleihs}, there exists a natural homomorphism $\varepsilon:H^2(S,\Z)\to H^2(S^{[n]},\Z)$ preserving the Beauville--Bogomolov form such that
\[H^2(S^{[n]},\Z)=\varepsilon(H^2(S,\Z))\oplus\Z e,\]
where $2e$ is the class of the irreducible divisor $E\subset S^{[n]}$ consisting of non-reduced subschemes. Moreover, one has $(e,e)=2-2n$ and therefore
\[H^2(X,\Z)\cong L_{K3}\oplus\langle2-2n\rangle=3U\oplus 2E_8(-1)\oplus\langle 2-2n\rangle=:\Ln\]
for any manifold $X$ of \kn-type.
\end{exm}

\subsection{The period map}
We now consider irreducible symplectic manifolds $X$ that are deformation equivalent to a fixed manifold $X_0$ and fix a lattice $L$ such that $H^2(X_0,\Z)\cong L$.
Let
\[\Omega_L:=\{\eta\in\Pb(L_\C):(\eta,\eta)=0,\ (\eta,\bar{\eta})>0\}\]
be the associated period domain.  

\begin{dfn} Let $X$ be an irreducible symplectic manifold as above.
\begin{enumerate}
\item
A \emph{marking} of $X$ is an isometry $\alpha:H^2(X,\Z)\to L$. 
\item
The \emph{period point} of a marked pair $(X,\alpha)$ is defined as 
\[P(X,\alpha):=\alpha(H^{2,0}(X))\in\Omega_{L}.\]
\end{enumerate}
\end{dfn}
 
Let $\pi:\calX\to\Def(X)$ be the Kuranishi family of $X=\pi^{-1}(0)$. It is a universal deformation since $H^0(X,\calT_X)=0$, and moreover unobstructed by a result of Bogomolov \cite{bogomolov}.
We assume that $\Def(X)$ is sufficiently small for the local system $R^2\pi_*\Z$ to be trivial. For any marking $\alpha_0:H^2(X,\Z)\to L$, there is a unique extension $\alpha:R^2\pi_*\Z\to L_{\Def(X)}$ where $L_{\Def(X)}$ is the constant sheaf of stalk $L$ on $\Def(X)$.

\begin{thm}[Local Torelli]
The period map $\Def(X)\to\Omega_L,\ t\mapsto P(X_t,\alpha_t)$ is an open embedding.
\end{thm}
\begin{proof}
\cite[Thm. 5]{beauvilleihs}
\end{proof}

Let $\frakM_{L}$ be the moduli space of marked pairs. As a consequence of the Local Torelli theorem, the period map $P:\frakM_{L}\to\Omega_L$ is a local isomorphism. 

\begin{thm}[Surjectivity of the period map]
For any connected component $\frakM_{L}^0\subset\frakM_{L}$, the restriction of the period map $P_0:\frakM_{L}^0\to \Omega_L$ is surjective.
\end{thm}
\begin{proof}
\cite[Thm. 8.1]{huybrechtsbasic}
\end{proof}

\begin{dfn} Let $X_1,X_2$ be irreducible symplectic manifolds.
\begin{enumerate}
\item
An isometry $g:H^2(X_1,\Z)\to H^2(X_2,\Z)$ is called a \emph{parallel transport operator}, if there exists a smooth family $\pi:\calX\to T$, two base points $t_1,t_2\in T$ with $\pi^{-1}(t_i)=X_i$ and a continuous path $\gamma:[0,1]\to T$ with $\gamma(0)=t_1,\ \gamma(1)=t_2$, such that the parallel transport in $R^2\pi_*\Z$ along $\gamma$ induces $g$. 
\item  A parallel transport operator $H^2(X,\Z)\to H^2(X,\Z)$ is called a \emph{monodromy operator} of $X$.
\end{enumerate}
\end{dfn}

The composition of monodromy operators is again a monodromy operator \cite[Footnote 3]{markmansurvey}. We denote by $\mon^2(X)\subset O(H^2(X,\Z))$ the group of monodromy operators of $X$, and by $\monh(X)\subset\mon^2(X)$ the group of monodromy operators which are Hodge isometries.

\smallskip

For any connected component $\frakM_{L}^0$ of the moduli space of marked irreducible symplectic manifolds, the subgroup 
\[\mon^2(\frakM_{L}^0):=\alpha\circ\mon^2(X)\circ\alpha^{-1}\subset O(L)\]
is independent of the choice of $(X,\alpha)\in\frakM_{L}^0$. It is the subgroup of $O(L)$ fixing $\frakM_{L}^0$ with respect to the action given by $\sigma(X,\alpha)=(X,\sigma\circ\alpha),\ \sigma\in O(L)$.

Moreover, for any manifold $X$ of \kn-type, $\mon^2(X)\subset O(H^2(X,\Z))$ is a normal subgroup by \cite[Thm. 1.2]{markmanintegral}. Thus, in this case the group $\mon^2(\frakM_{L_n}^0)$ does not depend on the choice of $\frakM_{L_n}^0$, and we denote this group by $\mon^2(L_n)\subset O(L_n)$.

The following lattice-theoretic description of monodromy operators was given by Markman. Let $O^+(L_n)\subset O(L_n)$ be the index two subgroup of isometries of real spinor norm 1 (for the definition of the spinor norm, we refer to \cite[\S 4]{markmansurvey} or \cite[\S 1]{ghs3}).

\begin{lem} \label{lem:monodromy}
The group $\mon^2(L_n)$ is the inverse image of $\{-1,1\}$ with respect to the natural homomorphism $O^+(L_n)\to O(L_n^*/L_n)$. In particular, one has $\mon^2(L_n)=O^+(L_n)$ if $n=2$ or if $n-1$ is a prime power.
\end{lem}
\begin{proof}
\cite[Lemma 9.2]{markmansurvey}
\end{proof}

Our main tool will be the Global Torelli theorem which was proved by Verbitsky \cite{verbitskytorelli}. We will only use the following Hodge-theoretic version which is due to Markman.

\begin{thm}[Global Torelli theorem] \label{thm:gt}
Let $X,Y$ be irreducible symplectic manifolds and $g:H^2(X,\Z)\to H^2(Y,\Z)$ a Hodge isometry which is a parallel transport operator. If $g$ maps a Kähler class to a Kähler class, then there exists a biholomorphic map $f:Y\to X$ with  $f^*=g$.
\end{thm}
\begin{proof}
\cite[Thm. 1.3]{markmansurvey}
\end{proof}

\subsection{Kähler-type chambers} \label{subsec:ktchambers}

In this subsection, we recall the description of the Kähler-type chambers given by Amerik--Verbitsky \cite{amerikverbitskyrational}. A similar result for the Kähler cone was shown by Mongardi \cite[Thm. 1.3]{mongardicone}.

\begin{dfn} Let $X$ be an irreducible symplectic manifold.
\begin{enumerate}
\item The \emph{positive cone} $\cx$ of $X$ is the connected component of
\[\widetilde{\calC}_X:=\{x\in H^{1,1}(X,\R):(x,x)>0\}\]
that contains the Kähler cone $\kx$ of $X$.
\item A \emph{Kähler-type chamber} of $X$ is a subset of $\cx$ of the form $g(f^*\calK_{\widetilde{X}})$ for some $g\in\monh(X)$ and a bimeromorphic map $f:X\dashrightarrow\widetilde{X}$ to an irreducible symplectic manifold $\widetilde{X}$.
\end{enumerate}
\end{dfn}

\begin{dfn}[{\cite[Def. 1.13]{amerikverbitskyrational}}]
A rational class $z\in H^{1,1}(X,\Q)$ with $(z,z)<0$ is called \emph{monodromy birationally minimal}, if there exists a bimeromorphic map $f:X\dashrightarrow\widetilde{X}$ and a monodromy operator $g\in\mon^2(X)$, such that the hyperplane $g(z)^\perp$ contains a face of $f^*\calK_{\widetilde{X}}$.
\end{dfn}

\begin{thm}[Amerik--Verbitsky] \label{thm:mbmdef}
Let $z\in H^{1,1}(X,\Z)$ be a monodromy biratio\-nal\-ly minimal class on $X$, and $(X',z')$ a deformation of $(X,z)$, such that $z'$ is of type $(1,1)$. Then $z'$ is monodromy birationally minimal.
\end{thm}
\begin{proof}
\cite[Thm. 2.16]{amerikverbitskycone}
\end{proof}

\begin{thm}[Amerik--Verbitsky] \label{thm:ktchambers}
The Kähler-type chambers of $X$ are the connected components of
\[\cx\setminus\bigcup_{z}z^\perp,\]
where the union is taken over all monodromy birationally minimal classes on $X$. 
\end{thm}
\begin{proof}
\cite[Thm. 6.2]{amerikverbitskyrational}
\end{proof}

Let $\Delta(X)\subset H^{1,1}(X,\Z)$ be the subset of all primitive integral classes which are monodromy birationally minimal. We call such classes \emph{wall divisors} (as in \cite{mongardicone}).
The set $\Delta(X)$ of wall divisors on manifolds of \kn-type has been explicitly determined for $n=2,3,4$ by Mongardi \cite{mongardicone}. We will only use the explicit description for $n=2$:

\begin{prop} \label{prop:wallk2}
A class $D\in H^{1,1}(X,\Z)$ is a wall divisor if and only if
\begin{enumerate}
\item $(D,D)=-2$, or
\item $(D,D)=-10$ and $(D,H^2(X,\Z))=2\Z$.
\end{enumerate}
\end{prop}
\begin{proof}
Hassett--Tschinkel \cite[Thm. 23]{hassetttschinkel} showed that every wall divisor is of this form, and Markman \cite[Thm. 1.11]{markmanprime} and Mongardi \cite[Prop. 2.12]{mongardicone} showed that every such class is a wall divisor.
\end{proof}


\section{Non-symplectic involutions}

\begin{dfn}
A \emph{non-symplectic involution} of an irreducible symplectic manifold $X$ is a biholomorphic involution $i:X\to X$ with $i^*\omega=-\omega$.
\end{dfn}

Let $i:X\to X$ be a non-symplectic involution and $i^*:H^2(X,\Z)\to H^2(X,\Z)$ be the induced isometry. 
For manifolds of \kn-type, the involution $i$ is determined by $i^*$:

\begin{thm} \label{thm:faithful}
Let $X$ be an irreducible symplectic manifold of $K3^{[n]}$-type and $f:X\to X$ an automorphism acting trivially on $H^2(X,\Z)$. Then $f=\id_X$.
\end{thm}
\begin{proof}
This was shown by Beauville \cite[Prop. 10]{beauvillerem} for the special case $X=S^{[n]}$ for some K3 surface $S$.
The general case follows from \cite[Cor. 6.9]{kaledinverbitsky} (see also \cite[\S 1.2]{markmanintegral}).
\end{proof}
An important invariant of the pair $(X,i)$ is the \emph{invariant sublattice}
\[H^2(X,\Z)^i=\{h\in H^2(X,\Z): i^*(h)=h\}\subset H^2(X,\Z).\]

\begin{exm}
Let $i:S\to S$ be a non-symplectic involution of a K3 surface $S$. This induces a non-symplectic involution of the Hilbert scheme of length $n$ subschemes $Z$ by
\[
\begin{array}{rccl}
i^{[n]}:&S^{[n]}&\to&S^{[n]}\\
&Z&\mapsto&i(Z).
\end{array}
\]
Such an automorphism of $S^{[n]}$ is called \emph{natural}.
Clearly, $i^{[n]}$ leaves the divisor $E$ globally invariant.
Furthermore, with respect to the natural embedding
$\varepsilon:H^2(S,\Z)\hookrightarrow H^2(S^{[n]},\Z)$ the restriction of $(i^{[n]})^*$ to $H^2(S,\Z)$ is given by $i^*$ \cite[Section 3]{boissieresarti}. Therefore, the invariant lattice of $i^{[n]}$ is given by
\[H^2(S^{[n]},\Z)^{i^{[n]}}=\varepsilon(H^2(S,\Z)^i)\oplus\Z e.\]
\end{exm}

We recall some well-known facts about non-symplectic involutions.

\begin{prop} \label{prop:invol}
Let $i:X\to X$ be a non-symplectic involution. Then
\begin{enumerate} \setlength\itemsep{2.5pt}
\item $(\omega,H^2(X,\Z)^i)=0$,
\item $H^2(X,\Z)^{i}\subset H^{1,1}(X,\Z)$,
\item $H^2(X,\Z)^{i}$ is hyperbolic,
\item $X$ is projective.
\end{enumerate}
\end{prop}
\begin{proof}
For every invariant class $h\in H^2(X,\Z)^i$, we have 
\[(\omega,h)=(i^*(\omega),i^*(h))=-(\omega,h),\]
which shows (i), since $H^2(X,\Z)$ is torsion-free. Together with \eqref{eq:ns} this implies (ii).
If $x\in H^2(X,\R)$ is a Kähler class, then $i^*(x)$ is a Kähler class and therefore
\[\tilde{x}:=x+i^*(x)\in H^2(X,\R)^i\]
is an invariant Kähler class. Since $(\tilde{x},\tilde{x})>0$, this implies that
\[H^2(X,\R)^i=H^2(X,\Z)^i\otimes\R\subset H^{1,1}(X,\R)\] 
is hyperbolic and hence (iii).
Part (iv) is a special case of \cite[Prop. 6]{beauvillerem}. Alternatively, it follows from (iii) and Huybrechts' projectivity criterion \cite[Thm. 3.11]{huybrechtsbasic}.
\end{proof}

\begin{dfn}
A \emph{family} $(\pi,I):\calX\to T$ of non-symplectic involutions over a connected smooth analytic space $T$ consists of
\begin{enumerate}
\item a smooth and proper family $\pi:\calX\to T$ of irreducible symplectic manifolds, and
\item a holomorphic involution $I:\calX\to\calX$ with $\pi\circ I=\pi$, such that for every $t\in T$, the induced involution $I_t:X_t\to X_t$ is non-symplectic. 
\end{enumerate}
\end{dfn}

\begin{dfn}
Let $i_1:X_1\to X_1$ and $i_2:X_2\to X_2$ be non-symplectic involutions. 
\begin{enumerate}
\item The pairs $(X_1,i_1)$ and $(X_2,i_2)$ are \emph{isomorphic}, if there exists an isomorphism $f:X_1\to X_2$ with $i_2\circ f=f\circ i_1$.
\item The pairs $(X_1,i_1)$ and $(X_2,i_2)$ are \emph{deformation equivalent}, if there exists a family $(\pi,I):\calX\to T$ of non-symplectic involutions and points $t_j\in T$ with $(X_{t_j},I_{t_j})\cong (X_j,i_j)$ for $j=1,2$. 
\end{enumerate}
\end{dfn}

The invariant sublattice of an involution is a deformation invariant in the following sense.

\begin{dfn} \label{dfn:latticetype}
$(X_1,i_1)$ and $(X_2,i_2)$ are of the same \emph{lattice type}, if there exists a parallel transport operator $g:H^2(X_1,\Z)\to H^2(X_2,\Z)$ with $g\circ i_1^*=i_2^*\circ g$.
\end{dfn}

\begin{prop} \label{prop:latticetype}
Let $i_1:X_1\to X_1$ and $i_2:X_2\to X_2$ be non-symplectic involutions such that $(X_1,i_1)$ and $(X_2,i_2)$ are deformation equivalent. Then $(X_1,i_1)$ and $(X_2,i_2)$ are of the same lattice type.
\end{prop}
\begin{proof}
This is a  consequence of Ehresmann's theorem, see for example \cite[Prop. 2.2]{ohashiwandel} or \cite[\S 4]{boissierecameresarti}.
\end{proof}

For K3 surfaces, the converse of Proposition \ref{prop:latticetype} is true. In fact, it suffices to assume that $i_1:S_1\to S_1$ and $i_2:S_2\to S_2$ are non-symplectic involutions of K3 surfaces with $H^2(S_1,\Z)^{i_1}\cong H^2(S_2,\Z)^{i_2}$ as lattices. Then $(S_1,i_1)$ and $(S_2,i_2)$ are deformation equivalent by \cite[Rem. 4.5.3]{nikulin1}.

For non-symplectic involutions of manifolds of \kn-type, being of the same lattice type is in general a stronger property than having isometric invariant lattices. However, we will see that even involutions of the same lattice type are not necessarily deformation equivalent.


\section{Local deformation space} \label{sec:localdef}

The local deformation theory of \kn-type manifolds with non-symplectic involutions has been described by Beauville \cite[Thm. 2]{beauvilleinv}. A more detailed discussion for automorphisms of prime order on irreducible symplectic manifolds is given in \cite[Section 4]{boissierecameresarti}. We briefly recall the facts.
 
Let $X$ be a manifold of \kn-type with a non-symplectic involution $i:X\to X$ and let
$\pi:\calX\to\Def(X)$ be the Kuranishi family of $X=\pi^{-1}(0)$. The involution $i$ on $X$ extends holomorphically to an involution $I:\calX\to\calX$, and by the universality of the Kuranishi family, this defines an action of $i$ on $\Def(X)$.
The deformation space $\Def(X)$ can be locally identified with $H^1(X,\tx)$ and the actions of $i$ on these spaces coincide under this identification. 
This shows that the invariant subspace $\Def(X,i):=\Def(X)^i$ is smooth.

Moreover, the symplectic form defines an isomorphism $\tx\to\Omega_X^1$, which maps the invariant subspace of $H^1(X,\tx)$ to the $(-1)$-eigenspace of $H^1(X,\Omega_X^1)$. 
In particular, the dimension of $\Def(X,i)$ is $21-\rk(H^2(X,\Z)^i)$.
Furthermore, the Kuranishi family restricts to a family
\[\pi':\calX'\to\Def(X,i),\] such that $I':=I|_{\calX'}$ preserves the fibres of $\pi'$.
\begin{exm} \label{exm:def}
Let $i:S\to S$ be a non-symplectic involution of a K3 surface $S$ and 
$i^{[n]}:S^{[n]}\to S^{[n]}$ be the natural involution. Any deformation of $(S,i)$ induces a deformation of $(S^{[n]},i^{[n]})$. On the other hand, we have 
\[H^{1,1}(S^{[n]})^{i^{[n]}}=\varepsilon(H^{1,1}(S)^i)\oplus\C e,\]
hence $\varepsilon$ maps $\Def(S,i)$ onto $\Def(S^{[n]},i^{[n]})$. Every small deformation of $(S^{[n]},i^{[n]})$ is induced by a deformation of $(S,i)$.
\end{exm}


\section{Period map}

From now on, we will only consider irreducible symplectic manifolds of \kn-type.

\begin{dfn} \label{dfn:admissiblelattice}
A sublattice $M\subset L_n$ is called \emph{admissible}, if 
\begin{enumerate}
\item $M$ is hyperbolic,
\item there exists an involution $\iota_M\in\mon^2(L_n)$ such that $M=(L_n)^{\iota_M}$.
\end{enumerate}
\end{dfn}

If $X$ is a manifold of \kn-type and $i:X\to X$ is a non-symplectic involution, then any marking
$\alpha:H^2(X,\Z)\to L_n$
 maps the invariant sublattice $H^2(X,\Z)^i\subset H^2(X,\Z)$ to some admissible sublattice $M\subset L_n$. 

In the case $n=2$, admissible sublattices have been classified by Boissi\`ere, Camere and Sarti in \cite{boissierecameresarti}. Moreover, it is shown that every such sublattice is isometric to the invariant
sublattice $H^2(X,\Z)^i\subset H^2(X,\Z)$ for some non-symplectic involution $i:X\to X$ of a $K3^{[2]}$-type manifold \cite[Prop. 8.2]{boissierecameresarti}.
We will see that the same is true for $n>2$.

We now fix a connected component $\frakM_{L_n}^0$ of the moduli space of marked manifolds of \kn-type and denote by $P_0:\frakM_{L_n}^0\to\Omega_L$ the restriction of the period map. This allows us to identify lattice types of non-symplectic involutions with $\mon^2(L_n)$-orbits of admissible sublattices.

\begin{dfn} \label{dfn:typem}
Let $M\subset L_n$ be an admissible sublattice with corresponding involution $\iota_M\in\mon^2(L_n)$ and let $i:X\to X$ be a non-symplectic involution of a \kn-type manifold $X$.
\begin{enumerate} \setlength\itemsep{.5em}
\item An \emph{($M$-)admissible marking} of $(X,i)$ is a marking $\alpha:H^2(X,\Z)\to L_n$ satisfying $(X,\alpha)\in\ml$
and $\alpha\circ i^*=\iota_M\circ\alpha$.
\item The pair $(X,i)$ is called \emph{of type} $M$, if there exists an $M$-admissible marking of $(X,i)$.
\end{enumerate}
\end{dfn}

\begin{rem}
If $(X,i)$ is of type $M$, then $H^2(X,\Z)^i$ is isometric to $M$. However, the converse is not true in general.
The definition of type $M$ depends a priori on the connected component $\frakM_{L_n}^0$ and on the embedding $M\subset L_n$. If $n=2$ or $n-1$ is a prime power, then the choice of $\frakM_{L_n}^0$ is irrelevant by Lemma \ref{lem:monodromy}. An example of non-isometric admissible sublattices $M,M'\subset L_2$ that are isometric as lattices is given in \cite[Example. 8.6]{boissierecameresarti}.
\end{rem}

Assume that $(X_1,i_1)$ is of type $M$ and $\alpha:H^2(X_1,\Z)\to L_n$ is an admissible marking. If $(X_2,i_2)$ is of the same lattice type as $(X_1,i_1)$, that is, there exists a parallel transport operator $g:H^2(X_1,\Z)\to H^2(X_2,\Z)$ with $g\circ i_1^*=i_2^*\circ g$, then $\alpha\circ g^{-1}$ is an admissible marking for $(X_2,i_2)$. In particular, a deformation of a pair of type $M$ is again of type $M$. As remarked before, the converse is true for K3 surfaces. Our goal is to give a lattice-theoretic description of the deformation types of pairs of type $M$.

Let
\[\calM_M:=\{(X,i): (X,i) \text{ is a pair of type } M\}/\cong.\]
For now,  we will consider $\calM_M$ only as a set. 

Let $(X,i)\in\calM_M$ and $\alpha:H^2(X,\Z)\to L_n$ be an admissible marking. Proposition \ref{prop:invol} (i) implies that
\[P_0(X,\alpha)\in\Omega_\Mp\subset\Omega_L.\] 
Consider the subgroup
\begin{align*}
\Gamma(M)&:=\{\sigma\in\mon^2(L_n):\sigma\circ \iota_M=\iota_M\circ\sigma\}\\
&\phantom{:}=\{\sigma\in\mon^2(L_n):\sigma(M)=M\}.
\end{align*}
If $(X,i)$ and $(Y,j)$ are of type $M$ 
with admissible markings $\alpha:H^2(X,\Z)\to L_n$ and $\beta:H^2(Y,\Z)\to L_n$, and $f:(X,i)\to (Y,j)$ is an isomorphism, then $f^*$ is a Hodge isometry, and therefore
\[P_0(X,\alpha)=\sigma(P_0(Y,\beta)),\quad\text{where}\quad\sigma:=\alpha\circ f^*\circ\beta^{-1}.\]
Since $f^*$ is a parallel transport operator and $(X,\alpha)$ and $(Y,\beta)$ belong to the same connected component of $\frakM_{L_n}$, we have $\sigma\in\mon^2(L_n)$. Furthermore, using $i^*\circ f^*=f^*\circ j^*$, we obtain
\begin{align*}
\iota_M\circ\sigma&=\iota_M\circ\alpha\circ f^*\circ\beta^{-1}=\alpha\circ i^*\circ f^*\circ\beta^{-1}\\
&=\alpha\circ f^*\circ j^*\circ\beta^{-1}=\alpha\circ f^*\circ\beta^{-1}\circ\iota_M=\sigma\circ\iota_M
\end{align*}
and hence $\sigma\in\Gamma(M)$. Thus the period map induces a map 
\begin{equation} \label{eq:period}
P_M:\calM_M\quad\longrightarrow\quad\Omega_\Mp/\gmp,
\end{equation}
where $\gmp\subset O(\Mp)$ is the image of the restriction homomorphism 
\[\Gamma(M)\to O(\Mp).\]

\begin{prop}
$\Gamma_\Mp\subset O(\Mp)$ is a finite index subgroup.
\end{prop}
\begin{proof}
It suffices to show that $\Gamma_\Mp$ contains the finite index subgroup 
\[\widetilde{O}^+(\Mp):=\widetilde{O}(\Mp)\cap O^+(\Mp)\subset O(\Mp),\] where $O^+(\Mp)\subset O(\Mp)$ is the index two subgroup of isometries of real spinor norm 1. By Lemma \ref{lem:stableisometry}, any isometry $\sigma\in \widetilde{O}^+(\Mp)$ extends to an isometry $\widetilde{\sigma}\in\widetilde{O}(L_n)$ with $\widetilde{\sigma}|_M=\id_M\in O^+(M)$ and hence $\widetilde{\sigma}\in O^+(L_n)$.
By Lemma \ref{lem:monodromy}, we have $\widetilde{\sigma}\in\mon^2(L_n)$, which shows $\widetilde{\sigma}\in\Gamma(M)$ and consequently $\sigma\in\Gamma_\Mp$. 
\end{proof}
Since $\sig(\Mp)=(2,r(\Mp)-2)$, the period domain $\Omega_\Mp$ consists of two connected components $\Omega_\Mp^+$ and $\Omega_\Mp^-$, each of which is isomorphic to a bounded symmetric domain.
Moreover, the finite index subgroup $\gmp\subset O(\Mp)$ acts properly discontinuously on $\Omega_\Mp$, and the quotient $\Omega_{M^\perp}/\gmp$ is a quasi-projective variety by \cite[Thm. 10.4 and Thm 10.11]{bailyborel}.

Assume that $(\pi,I):\calX\to T$ is a holomorphic family of involutions of type $M$. Then the holomorphicity of the ordinary period map implies that the induced map
\begin{align*}
T&\ \to\ \Omega_\Mp/\Gamma_\Mp\\ 
t&\ \mapsto\ P_M(X_t,I_t)
\end{align*}
is holomorphic.

For K3 surfaces, one has the following result due to Nikulin and Yoshikawa.

\begin{thm} \label{thm:k3}
The period map $P_M:\calM_M\to\Omega_\Mp/\Gamma_\Mp$ is injective and its image is a Zariski-open subset $\Omega_\Mp^0/\Gamma_\Mp$.
In particular, $\Omega_\Mp^0/\Gamma_\Mp$ is a coarse moduli space of pairs of type $M$.
\end{thm}
\begin{proof}
\cite[Thm. 1.8]{yoshikawa}
\end{proof}

For manifolds of \kn-type, we will see that even when the period map $P_M$ is restricted to involutions of a fixed deformation type, it need not be generically injective. 
However, once we fix a deformation type $\calK$, we will be able
to use a finer period map
\[P_{M,\calK}:\calM_{M,\calK}\to\Omega^+_\Mp/\Gamma_{\Mp,\calK}\] 
for some finite index subgroup $\Gamma_{\Mp,\calK}\subset\gmp$ and show that this map is generically injective.


\section{Stable invariant Kähler cone} \label{sec:sikc}

For a non-symplectic involution $i:X\to X$ let
\[\calC_X^i:=\{x\in\calC_X:i^*(x)=x\}\]
be the invariant positive cone and
\[\Delta^i(X):=\{D\in\Delta(X):i^*(D)=D\}\]
the set of invariant wall divisors of $(X,i)$.
It follows from Theorem \ref{thm:ktchambers} that the invariant Kähler cone $\calK_X^i=\calK_X\cap\cx^i$
of $(X,i)$ is contained in a connected component of
\begin{equation} \label{eq:sik}
\cx^i\setminus\bigcup_{D\in\Delta^i(X)} D^\perp.
\end{equation}

\begin{dfn}
The \emph{stable invariant Kähler cone} $\widetilde{\calK}_X^i$ of $(X,i)$ is the component of (\ref{eq:sik}) containing the invariant Kähler cone of $(X,i)$.
\end{dfn}

We will give a geometric interpretation of $\sik$ in Proposition \ref{prop:stablyample}. The distinction between invariant and non-invariant wall divisors is motivated by the following observations.

Assume that $i:X\to X$ is a non-symplectic involution and $f:X\dashrightarrow\widetilde{X}$ is a different birational model such that the induced birational involution
\[\tilde{i}:=f\circ i\circ f^{-1}:\widetilde{X}\to\widetilde{X}\]
is again biregular (see Example \ref{exm:comp} for a geometric realization of this situation). Then $f^*:H^2(\widetilde{X},\Z)\to H^2(X,\Z)$ is a parallel transport Hodge isometry satisfying $f^*\circ i^*=(\tilde{i})^*\circ f^*$, which implies $P_M(X,i)=P_M(\widetilde{X},\tilde{i})$.

If $\kx$ and $f^*\calK_{\widetilde{X}}$ are separated by a wall $D^\perp$ for an invariant wall divisor $D\in\Delta^i(X)$, then by Proposition \ref{prop:invol} (ii), the class $D$ remains of type $(1,1)$ for any deformation of the pair $(X,i)$. In this case the two birational models deform into different families. (At least locally; globally the families can be the same, as we will see in Example \ref{exm:four}.) In Section \ref{sec:defeq} we will see that the equivalence class of the stable invariant Kähler cone up to parallel transport determines the deformation type of a non-symplectic involution of a manifold of \kn-type.

If on the other hand $D$ belongs to $\Delta(X)\setminus\Delta^i(X)$, then the corresponding wall vanishes under some deformation of $(X,i)$, and the two pairs $(X,i)$ and $(\widetilde{X},\tilde{i})$ deform into the same family.
We will see that in this case one has $\calK_X^i\subsetneq\widetilde{\calK}_X^{i}$.
Our goal in this section is to identify a Zariski-closed subset $\calD_M'\subset\Omega_\Mp/\Gamma_\Mp$ such that $\widetilde{\calK}_X^{i}=\calK_X^i$ for every $(X,i)$ with $P_M(X,i)\not\in\calD_M'$. For this purpose, we will need the lattice-theoretic analogue of wall divisors.

\begin{prop} \label{prop:newample}
 For any connected component $\frakM_{L_n}^0$ of the moduli space of marked pairs of \kn-type, there exists a subset $\Delta(L_n)\subset L_n$ with the following properties:
\begin{enumerate}
\item For any $(X,\alpha)\in \frakM_{L_n}^0$, we have
\[\Delta(X)=\alpha^{-1}(\Delta(L_n))\cap H^{1,1}(X,\Z).\]
\item The group $\mon^2(L_n)$ acts on $\Delta(L_n)$ with a finite number of orbits.
\end{enumerate}
\end{prop}
\begin{proof} Let
\[\Delta(L_n):=\{\alpha(D):(X,\alpha)\in\frakM^0_{L_n}\text{ and } D\in\Delta(X)\}\subset L_n.\]
Let $(X,\alpha)\in\frakM_{L_n}^0$ and assume that $D=\alpha^{-1}(\beta(D'))\in H^{1,1}(X,\Z)$ for some marked manifold $(Y,\beta)\in\frakM_{L_n}^0$ and some wall divisor $D'\in\Delta(Y)$. Then $\alpha^{-1}\circ\beta$ is a parallel transport operator
and we have $D\in\Delta(X)$ by Theorem \ref{thm:mbmdef}. The other inclusion follows from the definition of $\Delta(L_n)$. This shows (i).

The group $\mon^2(L_n)$ clearly acts on $\Delta(L_n)$. The finiteness of orbits
is shown in \cite[Cor. 6.7]{amerikverbitskyrational}: by a result of Bayer--Hassett--Tschinkel \cite[Prop. 2]{bayerhassetttschinkel} there exists a constant $C_n>0$ such that any wall divisor $D\in\Delta(X)$ on a projective manifold $X$ of \kn-type satisfies $|(D,D)|<C_n$. In \cite{amerikverbitskyrational}, the authors extend this result to non-projective manifolds. Hence the claim follows from Lemma \ref{lem:orbits} and the fact that $\mon^2(L_n)\subset O(L_n)$ is a finite index subgroup.
\end{proof}

\begin{exm} By Proposition \ref{prop:wallk2} we have
\[\Delta(L_2)=\{\dd\in L_2:(\dd,\dd)=-2,\text{ or }(\dd,\dd)=-10,\ (\delta,L_2)=2\Z\}.\]
\end{exm}

Assume that $i:X\to X$ is a non-symplectic involution with $\calK_X^i\subsetneq\widetilde{\calK}_X^{i}$. Since $\widetilde{\calK}_X^{i}$ is connected, Theorem \ref{thm:ktchambers} implies that there exists a wall divisor $D\in\Delta(X)$ such that $\widetilde{\calK}_X^{i}\cap D^\perp\neq\emptyset$ and in particular $\calC_X^{i}\cap D^\perp\neq\emptyset$. By definition of the stable invariant Kähler cone, we have in fact $D\in\Delta(X)\setminus\Delta^i(X)$.
This, together with Lemma \ref{lem:delta}, motivates the following definition.

\begin{dfn}
We denote by $\Dp\subset L_n$ the set of elements $\dd\in L_n$ such that
\begin{align*}
\sig(M\cap\dd^\perp)&=(1,\rk(M)-2),\\
\sig(\Mp\cap\dd^\perp)&=(2,\rk(\Mp)-3).
\end{align*}
\end{dfn}

\begin{dfn} The \emph{positive cone} of $M$ is given by
\[\widetilde{\calC}_M:=\{x\in M_\R:(x,x)>0\}.\]
\end{dfn}

\begin{lem} \label{lem:delta}
For $\dd\in L_n$ the following properties are equivalent:
\begin{enumerate}
\item $\dd\in\Dp$

\item $\dd$ satisfies the following conditions:

\begin{enumerate}
\item $\dd\not\in M$,
\item $\dd\not\in\Mp$,
\item $\Omega_{\Mp}\cap\dd^\perp\neq\emptyset$,
\item $\widetilde{\calC}_M\cap\dd^\perp\neq\emptyset$.
\end{enumerate}

\item Let $\dd_M\in M_\Q$ and $\dd_\Mp\in M_\Q^\perp$ such that $\dd=\dd_M+\dd_\Mp$. Then
\[(\dd_M,\dd_M)<0,\qquad(\dd_\Mp,\dd_\Mp)<0.\]
\end{enumerate}
\end{lem}
\begin{proof}
First assume that $\dd\in\Dp$. 
Then (a) and (b) follow from $M\cap\dd^\perp\neq 0$ and $\Mp\cap\dd^\perp\neq 0$.
Since $M\cap\dd^\perp$ is hyperbolic, we have 
\[\widetilde{\calC}_M\cap\dd^\perp=\widetilde{\calC}_{M\cap\,\dd^\perp}\neq\emptyset,\]
and since $\Mp\cap\dd^\perp$ has two positive squares, we have 
\[\Omega_{\Mp}\cap\dd^\perp=\Omega_{\Mp\cap\,\dd^\perp}\neq\emptyset.\]
Conversely, assume that $\dd\in L_n$ satisfies (a)--(d). 
The sublattice $M\cap \dd^\perp\subset M$ is hyperbolic, parabolic or negative definite. 
The latter two cases are excluded by condition (d). 
Since $\dd\not\in\Mp$, this shows that $\sig(M\cap\dd^\perp)=(1,\rk(M)-2)$. 
Condition (c) implies that $\Mp\cap\dd^\perp$ has two positive squares and together with (a) we obtain $\dd\in\Dp$.

We now show the equivalence of (i) and (iii). 
If $\dd$ satisfies (iii), then the orthogonal decompositions
\begin{equation} \label{eq:orth}
\begin{aligned}
M_\Q&=(M_\Q\cap\dd^\perp)\oplus\Q\,\dd_M,\\
M^\perp_\Q&=(M^\perp_\Q\cap\dd^\perp)\oplus\Q\,\dd^\perp_M
\end{aligned}
\end{equation}
imply that $\dd\in\Dp$.
Conversely, if $\dd\in\Dp$, the lattices $M\cap\dd^\perp$ and $\Mp\cap\dd^\perp$ are non-degenerate,
and therefore we have $(\dd_M,\dd_M)\neq 0$ and $(\dd_\Mp,\dd_\Mp)\neq 0$. 
Thus, the decompositions (\ref{eq:orth}) hold and hence we have $(\dd_M,\dd_M)<0$
and $(\dd_\Mp,\dd_\Mp)<0$.
\end{proof}

\noindent For any sublattice $N\subset L_n$ let 
\[\Delta(N):=\Delta(L_n)\cap N.\]
Moreover, let \[\Dl:=\Dp\cap\Delta(L_n).\]

\begin{lem} \label{lem:discrloc}
The collections of hyperplanes 
\[\{\dd^\perp\subset\Omega_\Mp:\dd\in\Dl\}\] 
and 
\[\{\dd^\perp\subset\Omega_\Mp:\dd\in\Delta(\Mp)\}\] 
are locally finite in $\Omega_\Mp$.
\end{lem}
\begin{proof}
Since $\mon^2(L_n)$ acts on $\Delta(L_n)$, the group $\Gamma_\Mp$ acts on $\Delta(\Mp)$.
There are only finitely many possible values $(\dd,\dd)$ for $\dd\in\Delta(\Mp)\subset\Delta(L_n)$ by Proposition \ref{prop:newample} (ii), and since $\Gamma_\Mp\subset O(\Mp)$ is a finite index subgroup, Lemma \ref{lem:orbits} implies that $\Delta(\Mp)$ consists of finitely many $\gmp$-orbits.
The group $\gmp$ acts properly discontinuously on $\Omega_\Mp$, which means that the map
\begin{align*}
\Omega_\Mp\times\gmp\to&\ \Omega_\Mp\times\Omega_\Mp\\
(\eta,\sigma)\mapsto&\ (\eta,\sigma(\eta))
\end{align*}
is proper. In particular, every orbit $\gmp\cdot\dd^\perp\subset\Omega_\Mp$ is closed and hence a locally finite union of hyperplanes. 
This shows the first claim.

Now let $\dd\in\Dl$. We can write $2\dd=\dd_M+\dd_\Mp$, where 
\[\dd_M:=\dd+\iota_M(\dd)\in M,\qquad\dd_\Mp:=\dd-\iota_M(\dd)\in\Mp.\]
By Lemma \ref{lem:delta}, we have $(\dd_M,\dd_M)<0$ and $(\dd_\Mp,\dd_\Mp)<0$. 
Again, there is only a finite number of possible values for $(\dd,\dd)<0$, and since
\[4(\dd,\dd)=(\dd_M,\dd_M)+(\dd_\Mp,\dd_\Mp)\]
the same is true for $(\dd_\Mp,\dd_\Mp)$.
The group $\gmp$ acts on the set of such $\dd_\Mp$, and since $\dd^\perp=\dd_\Mp^\perp\subset\Omega_\Mp$, the same argument as above applies.
\end{proof}

By the preceding lemma, the subsets

\[\widetilde{\calD}_M:=\bigcup_{\dd\in\Delta(M^\perp)}\dd^\perp\subset\Omega_\Mp\]
and
\[\widetilde{\calD}'_M:=\bigcup_{\dd\in\Dl}\dd^\perp\subset\Omega_\Mp\]
are closed in $\Omega_\Mp$. They are invariant under $\gmp$ and their quotients
\[
\calD_M:=\widetilde{\calD}_M/\gmp\quad\text{and}\quad
\calD'_M:=\widetilde{\calD}'_M/\gmp
\]
are Zariski-closed subsets of $\Omega_\Mp/\gmp$. The significance of these divisors is explained by the following Proposition.

\begin{prop} \label{prop:stableperiod}
Let $(X,i)$ be a pair of type $M$.
\begin{enumerate}
\item $P_M(X,i)\not\in\calD_M$.
\item If $P_M(X,i)\not\in\calD'_M$, then we have $\sik=\calK_X^i$.
\end{enumerate}
\end{prop}
\begin{proof}
Let $\alpha:H^2(X,\Z)\to L_n$ be an admissible marking.
Assume first that $P_M(X,\alpha)\in\dd^\perp$, where $\dd\in\Delta(\Mp)$. Then $D:=\alpha^{-1}(\dd)$ is a wall divisor on $X$ which is orthogonal to the invariant lattice. This is impossible, since there exists an invariant ample class on $X$, and since ample classes are not orthogonal to any wall divisor by Theorem \ref{thm:ktchambers}. This shows (i).

Now assume that $\calK_X^i$ is strictly smaller than $\sik$. This implies that there exists an element $D\in\Delta(X)\setminus\Delta^i(X)$ such that $D^\perp$ has non-empty intersection with $\sik\subset\calC_X^i$. In particular, the element $\dd:=\alpha(D)$ satisfies $\dd^\perp\cap\widetilde{\calC}_M\neq\emptyset$. Furthermore we have $\dd\not\in\Mp$ by part (i) and $\dd\not\in M$ by the assumption $D\not\in\Delta^i(X)$. Finally, $P(X,\alpha)\in\Omega_\Mp\cap\dd^\perp$ shows that this intersection is non-empty, and we can apply Lemma \ref{lem:delta} to obtain $\dd\in\Dp$. Since $D\in\Delta(X)$, we have $\dd\in\Dl$ and hence $P_M(X,i)\in\calD'_M$. This shows (ii).
\end{proof}

We remark, that the converse of part (ii) is not true in general. In fact, as seen in the proof, the property $P_M(X,i)\in\calD'_M$ only implies the existence of a wall divisor in $\Delta(X)\setminus\Delta(X)^i$ whose orthogonal complement meets the invariant positive cone, rather than the stable invariant Kähler cone. Once we have discussed the problem of deformation equivalence, we will define a refined period map and a divisor $\calD_\calK$, which allows us to give a necessary and sufficient condition for $\sik=\calK_X^i$ in terms of the period map.


\section{Kähler-type chambers of $M$}

In this section we discuss the lattice-theoretic counterpart of the stable invariant Kähler cone, the Kähler-type chambers of $M$. We will use these in the next section to give a lattice-theoretic criterion for deformation equivalence.

\begin{dfn}
A \emph{Kähler-type chamber} of the lattice $M$ is a connected component of
\[\widetilde{\calC}_M\setminus\bigcup_{\dd\in\Delta(M)}\dd^\perp,\]
where $\Delta(M)=\Delta(L_n)\cap M$.
We denote the set of Kähler-type chambers of $M$ by $\kt(M)$.
\end{dfn}

If $(X,i)$ is a pair of type $M$ and $\alpha:H^2(X,\Z)\to L_n$ an admissible marking, then we have $\alpha(\Delta^i(X))=\Delta(M)$, and hence $\alpha(\sik)$ is a Kähler-type chamber of $M$.

\begin{dfn}
The stable invariant Kähler cones of two pairs $(X,i)$ and $(Y,j)$ of type $M$ are called \emph{isometric} if there exists a parallel transport operator $g:H^2(X,\Z)\to H^2(Y,\Z)$ satisfying 
\[j^*\circ g=g\circ i^*,\quad g(\sik)=\widetilde{\calK}^j_Y.\] 
In this case we write $\widetilde{\calK}_X^i\cong\widetilde{\calK}_Y^j$.
\end{dfn}

Let $\Gamma_M$ be the image of the homomorphism $\Gamma(M)\to O(M)$. 

\begin{prop}
$\Gamma_M\subset O(M)$ is a finite index subgroup.
\end{prop}
\begin{proof}
We will show that $\Gamma_M$ contains the group
$\widetilde{O}^+(M)=\widetilde{O}(M)\cap O^+(M)$,
from which the claim will follow.
By Lemma \ref{lem:stableisometry}, any $\sigma\in \widetilde{O}^+(M)$ extends to an isometry $\widetilde{\sigma}\in\widetilde{O}(L_n)$ satisfying $\widetilde{\sigma}|_\Mp=\id_\Mp\in O^+(\Mp)$ and hence $\widetilde{\sigma}\in O^+(L_n)$.
Using Lemma \ref{lem:monodromy}, we have $\widetilde{\sigma}\in\mon^2(L_n)$, which shows $\widetilde{\sigma}\in\Gamma(M)$ and consequently $\sigma\in\Gamma_M$.
\end{proof}

The group $\Gamma_M$ acts on $\Delta(M)$ and therefore on the Kähler-type chambers of $M$. We clearly have $\widetilde{\calK}_X^i\cong\widetilde{\calK}_Y^j$ if and only if 
\[[\alpha(\widetilde{\calK}_X^i)]=[\beta(\widetilde{\calK}_Y^j)]\in\kt(M)/\Gamma_M\]
for any and hence for all admissible markings $\alpha$ and $\beta$. In particular, we obtain a well-defined map
\begin{align*}
\rho:\ \calM_M\to&\ \kt(M)/\Gamma_M\\
(X,i)\mapsto&\ [\alpha(\widetilde{\calK}_X^i)].
\end{align*}

We will later show that the map $\rho$ is surjective. For this, we will need the following lemma.

\begin{lem} \label{lem:ktopen}
Any Kähler-type chamber of $M$ is an open subset of $\widetilde{\calC}_M$.
\end{lem}
\begin{proof}
Let $\calC_M\subset\widetilde{\calC}_M$ be one of the two connected components, and let $\Gamma_M^+\subset\Gamma_M$ and $O^+(M_\R)\subset O(M_\R)$ be the subgroups preserving $\calC_M$. The group $O^+(M_\R)$ acts transitively on
\[\mathbb{H}:=\{x\in\calC_M:(x,x)=1\}\]
and the stabilizer of $x\in\mathbb{H}$ is the compact group $O(x^\perp)$. By \cite[Lemma 3.1.1]{wolf}, the action of the discrete subgroup $\Gamma_M^+\subset O^+(M_\R)$ on 
\[\mathbb{H}\cong O^+(M_\R)/O(x^\perp)\] 
is properly discontinuous.
This implies that for $\dd\in\Delta(M)$, the $\Gamma_M^+$-orbit of the closed subset $\dd^\perp\subset\mathbb{H}$ is closed. Since $\Gamma_M^+\subset O(M)$ is a finite index subgroup, there is only a finite number of orbits.
\end{proof}


\section{Deformation equivalence} \label{sec:defeq}

The goal of this section is to show that two pairs $(X,i)$ and $(Y,j)$ of type $M$ are deformation equivalent if and only if their stable invariant Kähler cones are isometric. Moreover, we show that every Kähler-type chamber of $M$ can be realized as the stable invariant Kähler cone of some pair $(X,i)$, and thus obtain a purely lattice-theoretic characterization of the deformation types.

\bigskip

Let $\calK\in\kt(M)$ be a Kähler-type chamber. 
As a consequence of Lemma \ref{lem:ktopen}, there exists an integral class $h\in\calK$. 
By \cite[\S 4]{markmansurvey}, the component $\frakM_{L_n}^0$ and the connected component of $\widetilde{\calC}_M$ which contains $h$ (and therefore $\calK$) determine a connected component
$\Omega_{h^\perp}^+\subset\Omega_{h^\perp}$, such that for every 
\[(X,\alpha)\in\frakM^+_{h^\perp}:=P_0^{-1}(\Omega^+_{h^\perp})\]
we have $\alpha^{-1}(h)\in\cx$. 
Let $\frakM_{h^\perp}^a\subset\frakM^+_{h^\perp}$ be the set of marked pairs $(X,\alpha)$ such that $\alpha^{-1}(h)$ is ample.

\begin{lem} \label{lem:ampleopen}
$\frakM_{h^\perp}^a\subset\frakM^+_{h^\perp}$ is open.
\end{lem}
\begin{proof}
\cite[Cor. 7.3]{markmansurvey}
\end{proof}

Let $\Omega_\Mp^+$ be the connected component of $\Omega_\Mp$ which is contained in $\Omega_{h^\perp}^+$, and let
\[\frakM^+_\Mp:=P_0^{-1}(\Omega^+_\Mp).\]

We have $\alpha^{-1}(\calK)\subset\cx$ for every $(X,\alpha)\in\frakM^+_\Mp$. Let
\[\mk:=\{(X,\alpha)\in\frakM^+_\Mp:\calK\cap\alpha(\kx)\neq\emptyset\}.\]

\begin{lem} \label{lem:stableopen} $\mk\subset\frakM^+_\Mp$ is an open subset.
\end{lem}
\begin{proof}
Let $(X,\alpha)\in\mk$. Since $\alpha(\kx)\,\cap\,\calK\subset\calK$ is a non-empty open subset, there exists an integral element $h\in\calK$ such that $\alpha^{-1}(h)$ is ample.
Then 
\[\frakM_{h^\perp}^a\cap\frakM^+_\Mp\subset\mk\] is an open neighbourhood of $(X,\alpha)$ by Lemma \ref{lem:ampleopen}.
\end{proof}

\begin{prop} \label{prop:definv}
The isometry class of the stable invariant Kähler cone is invariant under deformation.
\end{prop}

\begin{proof}
Let $(\pi,I):\calX\to T$ be a family over a connected base $T$. For $s\in T$ let 
\[U_{s}=\{t\in T:\widetilde{\calK}_{X_{t}}^{I_{t}}\cong\widetilde{\calK}_{X_s}^{I_s}\}.\]
We claim that $U_{s}\subset T$ is open.
Let $U\subset T$ be a contractible open neighbourhood of $s$ and $\alpha:(R^2\pi_*\Z)|_U\to L_{U}$ a trivialization such that $\alpha_s$ is admissible for $(X_s,I_s)$. 
Then for every $t\in U$ the marking $\alpha_t$ is admissible for $(X_t,I_t)$, and we obtain
a holomorphic map $\phi:U\to\frakM^+_\Mp$, where $\frakM^+_\Mp=P_0^{-1}(\Omega^+_\Mp)$ is the connected component determined by $\calK:=\alpha_s(\widetilde{\calK}_{X_s}^{I_s})$ as described above.
By Lemma \ref{lem:stableopen}, the set
\[V:=\phi^{-1}(\mk)\subset T\]
is a non-empty open neighbourhood of $s$. For every $t\in V$ there is a Kähler class inside $\alpha_t^{-1}(\calK)$, and since $\widetilde{\calK}_{X_t}^{I_t}$ is determined by one invariant Kähler class, this implies $\widetilde{\calK}_{X_t}^{I_t}=\alpha_t^{-1}(\calK)$. Hence $\alpha_t^{-1}\circ\alpha_s$ is a parallel transport operator mapping $\widetilde{\calK}_{X_s}^{I_s}$ to $\widetilde{\calK}_{X_t}^{I_t}$.
This shows that $U_s\subset T$ is open and since $T=\bigcup_{s\in T} U_s$, we have $T=U_s$ for every $s\in T$.
\end{proof}

Let $h\in H^2(X,\Z)^i\subset H^{1,1}(X,\Z)$ and $\calL$ be a line bundle on $X$ with $c_1(\calL)=h$. Then we have $\Def(X,i)\subset\Def(X,\calL)$. For $t\in\Def(X,i)$, let $h_t:=c_1(\calL_t)$ where $(X_t,\calL_t)$ is the fibre over $t\in\Def(X,\calL)$ in the universal deformation of $(X,\calL)$.

The following Proposition gives a characterization of stably invariant ample classes which is similar to Markman's notion of stably prime exceptional classes \cite{markmanprime}.

\begin{prop} \label{prop:stablyample}
A class $h\in H^2(X,\Z)^i$ belongs to $\widetilde{\calK}_X^i$ if and only if there is an analytic subvariety $Z\subset\Def(X,i)$ of complex codimension 1 such that for $t\in\Def(X,i)\setminus Z$ the class $h_t$ is an invariant ample class of $(X_t,I_t)$.
\end{prop}
\begin{proof}
Let 
\[\pi:\calX\to\Def(X,i),\qquad I:\calX\to\calX\] 
be the universal deformation of $(X,i)=(\pi^{-1}(0),I_0)$ described in Section \ref{sec:localdef}.
We choose a trivialization $\alpha:R^2\pi_*\Z\to L_{\Def(X,i)}$ such that for every $t\in\Def(X,i)$, the marking $\alpha_t$ is admissible for $(X_t,I_t)$. Now the period map defines an open embedding $\Def(X)\subset\Omega_L$ such that 
\[\Def(X,i)=\Def(X)\cap\Omega_\Mp.\]
We have $(X,\alpha_0)\in\frakM_{\Mp,\calK}$, where $\calK:=\alpha_0(\sik)$, and by Lemma \ref{lem:stableopen} we can assume that $\Def(X,i)$ is sufficiently small such that $(X_t,\alpha_t)\in\mk$ for every $t\in\Def(X,i)$. As in the proof of Proposition \ref{prop:definv}, we see that
$\widetilde{\calK}_{X_t}^{I_t}=\alpha_t^{-1}(\calK)$.
By Lemma \ref{lem:discrloc}, the subset
\[Z:=\bigcup_{\dd\in\Dl}\dd^\perp\subset\Def(X,i)\] 
is a union of finitely many hyperplanes. For every $t\not\in Z$ we have $\calK_{X_t}^{I_t}=\alpha_t^{-1}(\calK)$ by Proposition \ref{prop:stableperiod}. Hence if $h\in H^2(X,\Z)^i$ belongs to $\widetilde{\calK}_X^i$, then the class \[h_t=\alpha_t^{-1}\circ\alpha_0(h)\in H^2(X_t,\Z)^{I_t}\]
is ample for every $t\not\in Z$.

Conversely let $h\in H^2(X,\Z)^i$ and assume that there exists a $t\in\Def(X,i)$ such that $h_t$ is ample. Then 
\[h=\alpha_0^{-1}\circ\alpha_t(h_t)\in\alpha_0^{-1}(\calK)=\sik\]
belongs to the stable invariant Kähler cone.
\end{proof}

In the following our aim is to show that also the converse of Proposition \ref{prop:definv} is true, that is, the isometry class of the stable invariant Kähler cone completely determines the deformation type.

Let $\calK$ be a Kähler-type chamber of $M$ and
$P_{\calK}:\mk\to\Omega^+_{\Mp}$
be the restriction of the period map.
Let $\Omega_\Mp^0:=\Omega_\Mp^+\setminus\widetilde{\calD}_M$. We have seen in Lemma \ref{lem:discrloc} that $\Omega_\Mp^0\subset\Omega_\Mp^+$ is an open subset.

\begin{lem} \label{lem:image}
The image of $P_{\calK}$ is $\Omega_\Mp^0$.
\end{lem}
\begin{proof}
By definition, for every $(X,\alpha)\in\mk$ there exists a Kähler class $x$ inside $\alpha^{-1}(\calK)\subset\alpha^{-1}(M_\R)$. 
Since $x\not\in D^\perp$ for every wall divisor $D\in\Delta(X)$, this shows that $P(X,\alpha)\not\in\widetilde{\calD}_M$.

Conversely, assume that $\eta\in\Omega_\Mp^+\setminus\widetilde{\calD}_M$. By the surjectivity of the period map, there exists a marked pair $(X,\alpha)\in\frakM_{L_n}^0$ with $P(X,\alpha)=\eta$. As noted before, since $\eta\in\Omega_\Mp^+$ and therefore $(X,\alpha)\in\frakM_\Mp^+$, we have $\alpha^{-1}(\calK)\subset\cx$. 

We claim that the cone $\alpha^{-1}(\calK)$ is not contained in the hyperplane $D^\perp$ for any $D\in\Delta(X)$. Indeed, since $\calK\subset M_\R$ is open, this would imply $\dd:=\alpha(D)\in M^\perp$, and therefore $\dd\in\Delta(\Mp)$. Then $P(X,\alpha)\in\dd^\perp\subset\widetilde{\calD}_M$ gives a contradiction.

Now it follows from Theorem \ref{thm:ktchambers} that $\alpha^{-1}(\calK)$ intersects a Kähler-type chamber of $X$. By definition, this means, that there exists a monodromy operator $g\in\monh(X)$ and a birational model $f:X\dashrightarrow\widetilde{X}$ such that 
\[\alpha^{-1}(\calK)\cap g(f^*\calK_{\widetilde{X}})\neq\emptyset\]
and therefore $\widetilde{\alpha}^{-1}(\calK)\cap\calK_{\widetilde{X}}\neq\emptyset$, where
\[\widetilde{\alpha}:=\alpha\circ g\circ f^*:H^2(\widetilde{X},\Z)\to L_n.\]
Since $g\circ f^*$ is a Hodge isometry and a parallel transport operator, we have 
\[P_0(\widetilde{X},\widetilde{\alpha})=P_0(X,\alpha)=\eta,\] and thus $(\widetilde{X},\widetilde{\alpha})\in\mk$ is a marked pair with $P_\calK(\widetilde{X},\widetilde{\alpha})=\eta$.
\end{proof}

Since $\mk\subset\frakM_\Mp^+$ is open, the period map restricts to a local isomorphism $P_\calK:\mk\to\Omega_\Mp^0$. We now want to use the path-connectedness of $\Omega_\Mp^0$ to show that $\mk$ is path-connected.  We will then define a family of involutions over $\mk$ containing any pair $(X,i)$ with $\rho(X,i)=[\calK]$.

\begin{lem} \label{lem:periodpath}
The space $\Omega_\Mp^0$ is path-connected.
\end{lem}
\begin{proof}
Let $\eta_1,\eta_2\in\Omega_\Mp^0\subset\Omega_\Mp^+$ and $\gamma:[0,1]\to\Omega_\Mp^+$ be a path connecting $\eta_1$ and $\eta_2$. By Lemma \ref{lem:discrloc}, for any $t\in[0,1]$, there exists a path-connected open neighbourhood $U_t\subset\Omega_\Mp^+$ of $\gamma(t)$ which intersects only finitely many hyperplanes $\dd^\perp$ for $\dd\in\Delta(\Mp)$. Let $V_1,\ldots,V_k$ be a finite subcovering of $\{U_t\}$ such that $\eta_1\in V_1,\ \eta_2\in V_k$ and 
\[V_i\cap V_{i+1}\neq\emptyset,\qquad i=0,\ldots,k-1.\] For any $i$, the set $V_i\setminus\widetilde{\calD}_M$ is the complement in $V_i$ of finitely many hyperplanes of real codimension 2 and therefore path-connected. Since $V_i\cap V_{i+1}\subset\Omega_\Mp^+$ is open, we have $V_i\cap V_{i+1}\cap\Omega_\Mp^0\neq\emptyset$, which shows the claim. 
\end{proof}

Locally, paths in $\Omega_\Mp^0$ can be lifted to paths in $\mk$ using the Local Torelli theorem. To connect these paths in $\mk$, we will need a dense subset of points which are unique in their fibres with respect to $P_\calK$. Let
\[\Omega_\Mp':=\Omega_\Mp\setminus\bigcup_{\delta\in L_n\setminus M}\delta^\perp\]
and
\[\mk':=\pk^{-1}(\Omega_\Mp').\]
For $(X,\alpha)\in\mk'$ we have $\alpha(H^{1,1}(X,\Z))=M$ and therefore $\alpha(\kx)=\calK$.

\begin{lem}
If $(X,\alpha)\in\mk'$, then
\[\pk^{-1}(\pk(X,\alpha))=\{(X,\alpha)\}.\]
\end{lem}
\begin{proof}
For $i=1,2$, let $(X_i,\alpha_i)\in\mk'$ with $\pk(X_1,\alpha_1)=\pk(X_2,\alpha_2)$. Then $\alpha_2^{-1}\circ\alpha_1:H^2(X_1,\Z)\to H^2(X_2,\Z)$ is a Hodge isometry and a parallel transport operator that maps $\calK_{X_1}$ onto $\calK_{X_2}$. 
By the Global Torelli theorem, $\alpha_2^{-1}\circ\alpha_1$ is induced by an isomorphism $f:X_2\to X_1$, which defines an isomorphism $(X_1,\alpha_1)\cong (X_2,\alpha_2)$ of marked pairs.
\end{proof}

The following Proposition is a generalization of \cite[Cor. 5.11]{markmanprime}, which contains the same statement for a rank 1 lattice $M=\Z h$ with $(h,h)>0$. In this case $\calK$ is the ray $\R_{>0}\cdot h$ and $\mk=\frakM_{h^\perp}^a$. We will use the same idea for the proof.

\begin{prop}
$\mk$ is path-connected.
\end{prop}
\begin{proof}
By the Local Torelli theorem  and Lemma \ref{lem:stableopen}, the surjective map 
\[\pk:\mk\to\Omega_\Mp^0\] is a local isomorphism.  
Let 
\[(X_0,\alpha_0),(X_1,\alpha_1)\in\mk\] and $\eta_i:=\pk(X_i,\alpha_i)$.
Let $\gamma:[0,1]\to\Omega_\Mp^0$ be a continuous path with $\gamma(i)=\eta_i$. It follows from the proof of Lemma \ref{lem:periodpath} that $\gamma$ can be chosen sufficiently generic such that 
\[T:=\gamma^{-1}(\Omega_\Mp')\subset[0,1]\]
is dense. For every $s\in P_\calK^{-1}(\gamma([0,1]))$ let $U_s\subset\mk$ be a path-connected open neighbourhood of $s$ which is mapped isomorphically onto an open subset of $\Omega_\Mp^0$. Then the sets $\pk(U_s)$ form an open covering of $\gamma([0,1])$, and we choose a finite subcovering 
\[V_i=\pk(U_{s_i}),\ i=0,\ldots,n+1.\]
We can assume that 
\[p_0:=(X_0,\alpha_0)\in U_{s_0},\ p_{n+1}:=(X_1,\alpha_1)\in U_{s_n}\]
and that $V_{i-1}\cap V_i\neq\emptyset$ for $i=1,\ldots,n$.  Let $t_0:=0,\ t_{n+1}:=1$,
\[t_i\in\gamma^{-1}(V_{i-1}\cap V_i)\cap T,\qquad i=1,\ldots,n,\]
and $p_i\in\mk'$ be the unique element of the fibre $\pk^{-1}(\gamma(t_i))$. Using the isomorphism $P_\calK|_{U_{s_i}}:U_{s_i}\to V_i$, the path $\gamma|_{{[t_{i},t_{i+1}]}}$ can be lifted to a path connecting $p_{i}$ with $p_{i+1}$.
\end{proof}

\begin{prop} \label{prop:inv}
For $(X,\alpha)\in\mk$, there exists a non-symplectic involution $i:X\to X$ with $i^*=\alpha^{-1}\circ\iota_M\circ\alpha$.
In particular, $(X,i)$ is of type $M$.
\end{prop}
\begin{proof}
Since $M$ is admissible, we have $\iota_M\in\mon^2(L_n)$, and therefore 
\[g:=\alpha^{-1}\circ\iota_M\circ\alpha:H^2(X,\Z)\to H^2(X,\Z)\]
is a monodromy operator. Furthermore, from $P(X,\alpha)\in\Omega_\Mp$ we obtain 
\[\alpha(H^{2,0}(X))\subset M^\perp\otimes\C\]
and hence that $g(\omega)=-\omega$. In particular, $g$ is a Hodge isometry. Finally, $g$ acts trivially on the chamber $\alpha^{-1}(\calK)$, which by assumption contains a Kähler class. By the Global Torelli theorem (Theorem \ref{thm:gt}), there exists an automorphism $i:X\to X$ with $i^*=g$. We have $i^*\circ i^*=\id_{H^2(X,\Z)}$, which by Theorem \ref{thm:faithful} shows that $i$ is an involution. The last assertion follows immediately from $i^*=\alpha^{-1}\circ\iota_M\circ\alpha$.
\end{proof}

\begin{thm} \label{thm:def}
Let $(X_1,i_1),(X_2,i_2)$ be two pairs of type $M$ with isometric stable invariant Kähler cones. Then $(X_1,i_1)$ and $(X_2,i_2)$ are deformation equivalent.
\end{thm}
\begin{proof}
Let $g:H^2(X_2,\Z)\to H^2(X_1,\Z)$ be a parallel transport operator with 
\[i_1^*\circ g=g\circ i_2^*,\qquad
g(\widetilde{\calK}_{X_2}^{i_2})=\widetilde{\calK}_{X_1}^{i_1}.\]
Let $\alpha_1:H^2(X_1,\Z)\to L_n$ be an admissible marking of $(X_1,i_1)$ and let 
\[\alpha_2:=\alpha_1\circ g:H^2(X_2,\Z)\to L_n.\]
For $j=1,2$, we have $(X_j,\alpha_j)\in\mk$, where 
\[\calK:=\alpha_1(\widetilde{\calK}_{X_1}^{i_1})=\alpha_2(\widetilde{\calK}_{X_2}^{i_2}).\] 
For every $(X,\alpha)\in\mk$ there exists an involution $i:X\to X$ such that 
\[i^*=\alpha^{-1}\circ \iota_M\circ\alpha\] by Proposition \ref{prop:inv}, which is unique by Theorem \ref{thm:faithful}. These involutions fit into a holomorphic family $(\pi,I):\calX\to \mk$. Indeed, let $U\subset\mk$ be a contractible open neighbourhood of $(X,\alpha)$ and 
\begin{align*}
\pi_U:\ &\calX_U\to U,\\ 
\alpha_U:\ &(R^2\pi_*\Z)|_U\to L_{U}
\end{align*}
be the universal family of marked manifolds defined over $U$. 
The involution $I_U:\calX_U\to\calX_U$, which is defined on each fibre as above, is holomorphic since it coincides with the universal deformation of $(X,i)$.
If $U,V$ are two such sets, we can glue $\calX_U$ and $\calX_V$ over $U\cap V$ and obtain a global family $\calX\to\mk$, since marked pairs do not admit non-trivial automorphisms by Theorem \ref{thm:faithful}. Moreover, the involutions $I_U$ and $I_V$ coincide over $U\cap V$, and we thus obtain a holomorphic involution $I:\calX\to\calX$ containing $(X_1,i_1)$ and $(X_2,i_2)$.
\end{proof}

\begin{thm} \label{thm:deftypes}
The map $\rho:\calM_M\to\kt(M)/\Gamma_M$ induces a bijection between deformation types of pairs of type $M$ and $\kt(M)/\Gamma_M$.
\end{thm}
\begin{proof}
By Theorem \ref{thm:def}, it only remains to show surjectivity. Let $\calK\in\kt(M)$ be a Kähler-type chamber. By Lemma \ref{lem:image}, the set $\mk$ is non-empty, and by Proposition \ref{prop:inv}, for any $(X,\alpha)\in\mk$ there exists an involution $i:X\to X$ with $i^*=\alpha^{-1}\circ\iota_M\circ\alpha$. The isometry $i^*$ acts trivially on $\alpha^{-1}(\calK)$, which implies $\alpha^{-1}(\calK)\cap\kx^i=\alpha^{-1}(\calK)\cap\kx$. Since $(X,\alpha)\in\mk$, this intersection is non-empty,
and since $\kx^i\subset\sik$, this shows $\alpha(\sik)=\calK$.
\end{proof}

\begin{exm} \label{exm:comp}
We apply Theorem \ref{thm:deftypes} to an example given by Ohashi--Wandel \cite{ohashiwandel}.
Let $\pi:S\to\Pb^2$ be a K3 surface which is a double plane branched over a smooth sextic. Let $i:S\to S$ be the covering involution
and $i^{[2]}:S^{[2]}\to S^{[2]}$ the natural involution. 
The fixed locus of $i^{[2]}$ contains the plane
\[P:=\{[s,i(s)]\in S^{[2]}:s\in S\}\cong S/i\cong\Pb^2.\]
The authors consider the Mukai flop $f:S^{[2]}\dashrightarrow X$ obtained by replacing $P$ by the dual projective plane $P^*=|\calO_P(1)|$ and show that the induced birational involution 
\[j:=f\circ i^{[2]}\circ f^{-1}:X\to X\]
is biregular.
The invariant lattice of $i$ is isometric to $\langle2\rangle$ and therefore
that of the natural
involution $i^{[2]}$ is given by
\[H^2(S^{[2]},\Z)^{i^{[2]}}=\Z h\oplus\Z e\cong\langle2\rangle\oplus\langle-2\rangle,\]
where $e$ is half the class of the exceptional divisor and $h$ is the image of a primitive invariant ample class on $S$ under the natural map 
$H^2(S,\Z)\hookrightarrow H^2(S^{[2]},\Z)$.
Hence $(S^{[2]},i^{[2]})$ is of type
\[M:=\varepsilon(\langle2\rangle)\oplus\langle-2\rangle\subset L_2,\]
where
\[\varepsilon:L_{K3}\hookrightarrow L_2=L_{K3}\oplus\langle-2\rangle\]
is the natural inclusion. Ohashi and Wandel show that every pair of type $M$ can be deformed into $(S^{[2]},i^{[2]})$ or into $(X,j)$.

The invariant wall divisors of $(S^{[2]},i^{[2]})$ are given by 
\[\Delta^{i^{[2]}}(S^{[2]})=\pm\{e,2h+3e,2h-3e\}\]
and divide the invariant positive cone of $(S^{[2]},i^{[2]})$ into four chambers shown in Figure~\ref{fig:comp}.

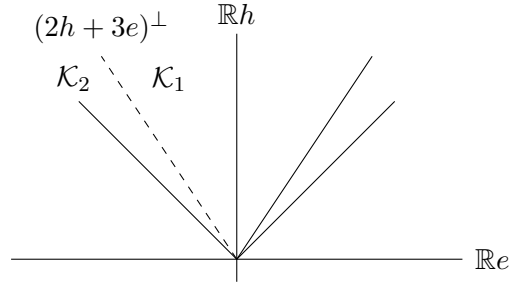
\begin{figure}[h!]
\begin{center}
\begin{tikzpicture}[scale=.3]
\draw (-10,0) -- (10,0) node[right=1pt] {$\R e$};
\draw (0,-1) -- (0,10) node[anchor=south] {$\R h$};
\draw (-7,7) -- (0,0) -- (7,7);
\draw[dashed] (0,0) -- (-6,9) node[anchor=north east] {$\calK_2$};
\draw (0,0) -- (6,9);
\draw (-3,8) node {$\calK_1$};
\draw (-6,4);
\draw (-6, 9) node[above=1pt] {$(2h+3e)^\perp$};
\end{tikzpicture}
\caption{Decomposition of the invariant positive cone}
\label{fig:comp}
\end{center}
\end{figure}

By a result of Bayer--Macr\`i \cite[Lemma 13.3]{bayermacri}, the invariant Kähler cone of $S^{[2]}$ is equal to $\calK_1$.
On the other hand, by the proof of \cite[Cor. 2.11]{ohashiwandel}, we have $\calK_2=f^*\kx^j$. 

Since any isometry of $\Z h\oplus\Z e\cong\langle2\rangle\oplus\langle-2\rangle$ maps $h$ to $\pm h$ and $e$ to $\pm e$, the cones $\calK_1$ and $\calK_2$ are not isometric, and hence the pairs $(S^{[2]},i^{[2]})$ and $(X,j)$ are not deformation equivalent.
This answers a question from \cite{ohashiwandel}. There are exactly two deformation types of pairs of type $M$, one of which contains all natural involutions. 

\end{exm}


\section{Moduli spaces} \label{sec:mod}

By Theorem \ref{thm:deftypes}, the deformation equivalence classes of $\calM_M$ are given by
\[\calM_{M,\calK}:=\{(X,i)\in\calM_M:\rho(X,i)=[\calK]\},\]
where $[\calK]\in\kt(M)/\Gamma_M$. In this section, we want to replace the period map $P_M$ by a finer period map $P_{M,\calK}$ which maps $\calM_{M,\calK}$ generically injectively onto a quasi-projective variety. For the rest of this section, we fix a representative $\calK$ of $[\calK]$ and denote by $\Omega_\Mp^+$ the connected component determined by $\frakM_{L_n}^0$ and $\calK$.

\begin{dfn}
Let $(X,i)\in\calM_{M,\calK}$. A marking $\alpha:H^2(X,\Z)\to L_n$ is called \emph{admissible for $\calK$}, if it is admissible for $M$ and furthermore satisfies $\alpha(\sik)=\calK$.
\end{dfn}

By definition of $\calM_{M,\calK}$, for any pair $(X,i)\in\calM_{M,\calK}$ there exists a marking which is admissible for $\calK$. 
Moreover, any two such markings differ by an element of 
\[\Gamma(\calK):=\{\sigma\in\Gamma(M):\sigma(\calK)=\calK\}\subset O(L_n).\]

Let $\Gamma_{\Mp,\calK}$ be the image of the restriction homomorphism $\Gamma(\calK)\to O(\Mp)$.
Since 
\[\Gamma(\calK)\subset\mon^2(L_n)\subset O^+(L_n),\] and since the image of $\Gamma(\calK)\to O(M)$ is contained in $O^+(M)$, we have
\[\Gamma_{\Mp,\calK}\subset O^+(\Mp),\] 
where $O^+(\Mp)$ is the subgroup of isometries with real spinor norm $+1$, or equivalently, the subgroup of isometries preserving $\Omega_\Mp^+$.

\begin{prop}
$\Gamma_{\Mp,\calK}$ is a finite index subgroup of $O^+(\Mp)$.
\end{prop}
\begin{proof}
By Lemma \ref{lem:stableisometry}, any isometry $\sigma\in \widetilde{O}^+(\Mp)$ extends to an isometry $\widetilde{\sigma}\in\widetilde{O}(L_n)$ with $\widetilde{\sigma}|_M=\id_M\in O^+(M)$ and hence $\widetilde{\sigma}\in O^+(L_n)$.
Using Lemma \ref{lem:monodromy}, this implies $\widetilde{\sigma}\in\mon^2(L_n)$ and consequently $\widetilde{\sigma}\in\Gamma(M)$. Since $\widetilde{\sigma}$ acts trivially on $M$, we have $\widetilde{\sigma}\in\Gamma(\calK)$ and therefore $\sigma\in\Gamma_{\Mp,\calK}$. This shows that $\widetilde{O}^+(\Mp)\subset\Gamma_{\Mp,\calK}$ and thus that $\Gamma_{\Mp,\calK}\subset O^+(\Mp)$ has finite index.
\end{proof}

As before, the quotient $\Omega^+_{M^\perp}/\Gamma_{\Mp,\calK}$ is a quasi-projective variety, and we have a well-defined map
\begin{equation} \label{eq:period2}
\begin{aligned} 
P_{M,\calK}:\calM_{M,\calK}&\ \to\ \Omega^+_\Mp/\Gamma_{\Mp,\calK}\\
(X,i)&\ \mapsto\ [P(X,\alpha)],
\end{aligned}
\end{equation}
where $\alpha:H^2(X,\Z)\to L_n$ is any marking which is admissible for $\calK$.

Assume that $(\pi,I):\calX\to T$ is a deformation of $(X,i)=(\pi^{-1}(0),I_0)$. Let $U\subset T$ be a contractible open neighbourhood of $0$ and $\alpha:(R^2\pi_*\Z)|_U\to L_U$ a trivialization such that $\alpha_0$ is admissible for $\calK$. Then for every $t\in U$ the marking $\alpha_t$ is admissible for $\calK$, as shown in the proof of Proposition \ref{prop:definv}. Since the ordinary period map is holomorphic, by \cite[Thm. 3.27]{schmid} this shows that the induced map 
\begin{align*}
T&\ \to\ \Omega^+_\Mp/\Gamma_{\Mp,\calK}\\
t&\ \mapsto\ P_{M,\calK}(X_t,I_t)
\end{align*}
is holomorphic.

Our goal is to show that the map $P_{M,\calK}$ is generically injective. However, we will see in Example \ref{exm:nonsep} that in general $P_{M,\calK}$ is not injective and $\calM_{M,\calK}$ does not admit a structure as a Hausdorff moduli space.
We therefore restrict to the following class of pairs $(X,i)$ in order to obtain a quasi-projective (and in particular Hausdorff) moduli space.

\begin{dfn}
A pair $(X,i)$ of type $M$ is called \emph{simple}, if $\sik=\kx^i$.
\end{dfn}

Let
\[\Delta(\calK):=\{\dd\in\Dl:\dd^\perp\cap\calK\neq\emptyset\}.\]
Since $\Delta(\calK)\subset\Dl$, it follows from Lemma \ref{lem:discrloc}, that the collection of hyperplanes 
\[\{\dd^\perp\subset\Omega^+_\Mp,\ \dd\in\Delta(\calK)\}\] 
is locally finite and thus their union
\[\widetilde{\calD}_\calK:=\bigcup_{\dd\in\Delta(\calK)}\dd^\perp\subset\Omega^+_\Mp\]
is closed. Furthermore, $\widetilde{\calD}_\calK$ is invariant under $\Gamma_{\Mp,\calK}$ and the quotient
\[\calD_\calK:=\widetilde{\calD}_\calK/\Gamma_{\Mp,\calK}\subset\Omega^+_{M^\perp}/\Gamma_{\Mp,\calK}\]
is Zariski-closed. Hence 
\[\calM_{M,\calK}^0:=(\Omega_\Mp^+/\Gamma_{\Mp,\calK})\setminus(\calD_M\cup\calD_\calK)\]
is a quasi-projective variety. 

\begin{prop} \label{prop:simple}
Let $(X,i)$ be of deformation type $\calK$. Then $(X,i)$ is simple if and only if $P_{M,\calK}(X,i)\not\in\calD_\calK$.
\end{prop}
\begin{proof}
Let $\alpha:H^2(X,\Z)\to L_n$ be a marking which is admissible for $\calK$.
If $(X,i)$ is not simple, then in the proof of Proposition \ref{prop:stableperiod} it was shown that there exists a wall divisor $D\in\Delta(X)$ with $D^\perp\cap\sik\neq\emptyset$ such that $P(X,\alpha)\in\dd^\perp$ where $\dd:=\alpha(D)\in\Dl$. This shows in fact that $\dd\in\Delta(\calK)$, and hence one implication.

Conversely, assume that $P(X,\alpha)\in\dd^\perp$, where $\dd\in\Delta(\calK)$. Then the class $D:=\alpha^{-1}(\dd)\in\Delta(X)$ is a wall divisor with $D^\perp\cap\sik\neq\emptyset$, which implies that $(X,i)$ is not simple.
\end{proof}

\begin{thm} \label{thm:modnat}
$\calM_{M,\calK}^0$ is a coarse moduli space for simple pairs of type $M$ and of deformation type $[\calK]$.
\end{thm}

\begin{proof}
Let $\eta\in\Omega_\Mp^0=\Omega^+_{\Mp}\setminus\widetilde{\calD}_M$. By Lemma \ref{lem:image} there exists a marked pair $(X,\alpha)\in\mk$ with $P(X,\alpha)=\eta$, and by Proposition \ref{prop:inv}, there exists a non-symplectic involution $i:X\to X$ with $i^*=\alpha^{-1}\circ\iota_M\circ\alpha$. In the proof of Theorem \ref{thm:deftypes} it was shown that $\alpha(\sik)=\calK$.
Together with Proposition \ref{prop:simple}, this shows that the period map $P_{M,\calK}$ restricts to a surjective map
\[P_{M,\calK}:\{(X,i)\in\calM_{M,\calK}: (X,i)\text{ is simple}\}\to\calM_{M,\calK}^0.\]
It remains to show that this map is injective.
Assume that $(X_1,i_1),(X_2,i_2)$ are two simple pairs with $P_{M,\calK}(X_1,i_1)=P_{M,\calK}(X_2,i_2)$. Let 
\[\alpha_j:H^2(X_j,\Z)\to L_n,\quad j=1,2\]
be markings that are admissible for $\calK$.
By assumption, there exists an isometry $\tau\in\Gamma(\calK)$ such that $\tau(P_0(X_2,\alpha_2))=P_0(X_1,\alpha_1)$. This means, that
\[g:=\alpha_1^{-1}\circ\tau\circ\alpha_2:H^2(X_2,\Z)\to H^2(X_1,\Z)\] is a Hodge isometry. Since $(X_1,\alpha_1),(X_2,\alpha_2)\in\frakM_{L_n}^0$ and $\tau\in\mon^2(L_n)$, it is a parallel transport operator. Furthermore, we have
\[g(\calK_{X_2}^{i_2})=g(\widetilde{\calK}_{X_2}^{i_2})=\widetilde{\calK}_{X_1}^{i_1}=\calK^{i_1}_{X_1}.\]
Since the invariant Kähler cones are non-empty, $g$ maps a Kähler class to a Kähler class, and 
by the Global Torelli theorem, there exists an isomorphism $f:X_1\to X_2$ with $f^*=g$. Moreover,
\begin{align*}
(i_2\circ f)^*&=f^*\circ i_2^*=\alpha_1^{-1}\circ\tau\circ\alpha_2\circ i_2^*=\alpha_1^{-1}\circ\tau\circ\iota_M\circ\alpha_2\\
&=\alpha_1^{-1}\circ\iota_M\circ\tau\circ\alpha_2=i_1^*\circ\alpha_1^{-1}\circ\tau\circ\alpha_2=i_1^*\circ f^*\\
&=(f\circ i_1)^*,
\end{align*}
which by Theorem \ref{thm:faithful} implies that $i_2\circ f=f\circ i_1$ and hence that  
\[f:(X_1,i_1)\isoarrow(X_2,i_2).\qedhere\]
\end{proof}

We will now describe the pairs $(X,i)$ that are not unique in their fibre with respect to $P_{M,\calK}$.

\begin{dfn}
Two non-isomorphic pairs $(X,i)$ and $(\widetilde{X},\tilde{i})$ are called \emph{inseparable}, if their universal deformations 
\[(\pi,I):\calX\to\Def(X,i),\quad(\widetilde{\pi},\widetilde{I}):\widetilde{\calX}\to\Def(\widetilde{X},\widetilde{i})\] (considered as germs) contain isomorphic fibres.
\end{dfn}

\begin{prop} \label{prop:insep}
Suppose that $(X_1,i_1)$ and $(X_2,i_2)$ are two non-isomorphic pairs of type $M$ and deformation type $\calK$ with 
\[P_{M,\calK}(X_1,i_1)=P_{M,\calK}(X_2,i_2).\] Then $(X_1,i_1)$ and $(X_2,i_2)$ are inseparable.
\end{prop}
\begin{proof}
By the assumption $P_{M,\calK}(X_1,i_1)=P_{M,\calK}(X_2,i_2)$, there exist admissible markings $\alpha_1:H^2(X_1,\Z)\to L_n$ and $\alpha_2:H^2(X_2,\Z)\to L_n$ with 
\[\eta:=P_0(X_1,\alpha_1)=P_0(X_2,\alpha_2).\] These induce embeddings 
\[\Def(X_1,i_1),\,\Def(X_2,i_2)\subset\Omega_\Mp,\qquad j=1,2\] as open neighbourhoods of $\eta$. Let $\eta'$ be any point inside the open subset 
\[(\Def(X_1,i_1)\cap\Def(X_2,i_2))\setminus\widetilde{\calD}_\calK.\]
By Theorem \ref{thm:modnat}, there is a unique pair $(X',i')$ which is in the fibre over $\eta'$ in the universal deformation of both $(X_1,i_1)$ and $(X_2,i_2)$.
\end{proof}

We will now consider an example where $\calK_X^i$ is strictly smaller than $\sik$, and the chambers of $\sik$ correspond to the invariant Kähler cones of inseparable birational models.

\begin{exm} \label{exm:nonsep}
Let $n=2$ and $e_1,e_2$ and $f_1,f_2$ be standard bases for the first two hyperbolic planes of $L_2=3U\oplus2E_8(-1)\oplus\Z e$, where $(e,e)=-2$. We consider the involution acting by $\iota_M(e_i)=f_i$ on $2U$ and as $-1$ on its orthogonal complement $U\oplus 2E_8(-1)\oplus\Z e$. Then the invariant lattice and the coinvariant
lattice are given by
\begin{align*}
M&=\Z(e_1+f_1)+\Z(e_2+f_2)\cong U(2)\\
M^\perp&=(\Z(e_1-f_1)+\Z(e_2-f_2))\oplus U\oplus 2E_8(-1)\oplus\Z e\\
&\cong U(2)\oplus U\oplus 2E_8(-1)\oplus\langle-2\rangle.
\end{align*}
Let $\delta:=2e_1-2e_2+e$. We have $(\delta,\delta)=-10,\ (\delta,L_2)=2\Z$ and hence $\delta\in\Delta(L_2)$.
Moreover, we can write $\delta=\delta_M+\delta_\Mp$ with
\[
\begin{array}{lll}
\delta_M&=e_1+f_1-e_2-f_2&\in M\\
\delta_\Mp&=e_1-f_1-e_2+f_2+e&\in\Mp.
\end{array}
\]
Since $(\delta_M,\delta_M)=-4<0$ and $(\delta_\Mp,\delta_\Mp)=-6<0$, we have $\delta\in\Delta_M(L_2)$. By Lemma \ref{lem:delta}, the intersection $\Omega_\Mp\cap\delta^\perp$ is non-empty and for a generic period point
$\eta\in\delta^\perp\subset\Omega_\Mp$, we have
\[
L_2\cap\eta^\perp=M+\Z\dd=M\oplus\Z\dd_\Mp\cong U(2)\oplus\langle-6\rangle\]
and in particular $\eta\not\in\calD_M$.

Using the surjectivity of the period map, let $(\widetilde{X},\widetilde{\alpha})\in\frakM_{L_2}^0$ be a marked manifold of $K3^{[2]}$-type with $P(\widetilde{X},\widetilde{\alpha})=\eta$. Let $\calC_M\subset\widetilde{\calC}_M$ be the connected component which is contained in $\widetilde{\alpha}(\calC_{\widetilde{X}})$, and 
let $\calC_1,\calC_2$ be the chambers of $\calC_M$ which are separated by $\delta^\perp$.
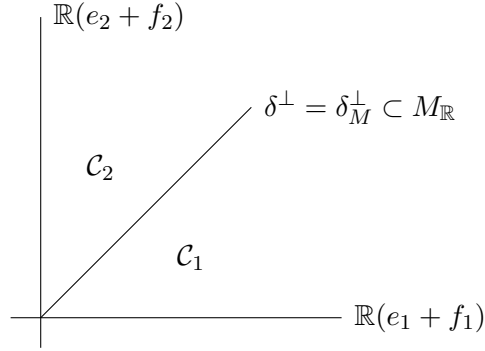
\begin{figure}[h] \label{fig:nonsep}
\begin{center}
\begin{tikzpicture}[scale=.4]
\draw (-1,0) -- (10,0) node[right=1pt] {$\R(e_1+f_1)$};
\draw (0,-1) -- (0,10) node[right=1pt] {$\R(e_2+f_2)$};
\draw (0,0) -- (7,7) node[right=1pt] {$\dd^\perp=\dd_M^\perp\subset M_\R$};
\draw (5,2) node {$\calC_1$};
\draw (2,5) node {$\calC_2$}; 
\end{tikzpicture}
\caption{Decomposition of the cone $\calC_M$}
\label{fig:cones}
\end{center}
\end{figure}
The fact that $\eta\not\in\calD_M$ implies that $\widetilde{\alpha}^{-1}(\calC_M)$ and hence $\widetilde{\alpha}^{-1}(\calC_j),\ j=1,2$ intersect Kähler-type chambers of $\widetilde{X}$ (for details, see the proof of Lemma \ref{lem:image}). By definition of Kähler-type chamber, 
there exist two marked manifolds $(X_1,\alpha_1),(X_2,\alpha_2)\in\frakM_{L_2}^0$ with $P(X_j,\alpha_i)=\eta$ and \[\emptyset\neq\alpha_j(\calK_{X_j})\cap M_\R\subset\calC_j.\] 
(In fact, one can see that equality holds, but we do not need this.) Since $M\cong U(2)$ does not contain elements of length $-2$ or $-10$, the cone $\calK:=\calC_M$ is a Kähler-type chamber of $M$ and $(X_j,\alpha_j)\in\frakM_{\Mp,\calK}$ for $j=1,2$.
By Proposition \ref{prop:inv}, for $j=1,2$ there exist non-symplectic involutions $i_j:X_j\to X_j$ such that $(X_j,i_j)\in\calM_{M,\calK}$ and $P_{M,\calK}(X_j,i_j)=[\eta]$. We will now show that $(X_1,i_1)$ and $(X_2,i_2)$ are not isomorphic. Proposition~\ref{prop:insep} then shows that $(X_1,i_1)$ and $(X_2,i_2)$ are inseparable.

Assume that $f:(X_1,i_1)\to (X_2,i_2)$ is an isomorphism and consider
\[\psi:=\alpha_1\circ f^*\circ\alpha_2^{-1}:L_2\to L_2.\]
We have $\psi\in\Gamma(M)$ and denote by $\psi_M\in O(M)$ and $\psi_\Mp\in O(\Mp)$ its restrictions. Since $\eta$ is generic and $f^*$ is a Hodge isometry, we have $\psi_\Mp=\pm\id_\Mp$ and in particular $\psi\in\widetilde{O}(\Mp)$.
By Proposition \ref{prop:embeddings}, the isomorphism $\gamma:H_M\to H_{\Mp}$
conjugates $\psi_{\Mp}|_{H_\Mp}$ to $\psi_M|_{H_M}$. On the other hand, since $f^*$ maps an invariant Kähler class to an invariant Kähler class, we have $\psi_M(\calC_1)=\calC_2$, and therefore $\psi_M$ acts non-trivially on $H_M=A_M$, which gives a contradiction.

Note that $\calK_{X_j}^{i_j}\subset\alpha_j^{-1}(\calC_j)\subsetneq\alpha_j^{-1}(\calC_M)=\widetilde{\calK}_{X_j}^{i_j},\ j=1,2$, which means that $(X_j,i_j)$ is not simple.
\end{exm}

The following example shows that the groups $\Gamma_{\Mp,\calK}$ can be different for different deformation classes $\calM_{M,\calK}$ of $\calM_{M}$.

\begin{exm} \label{exm:four}
We again consider a double plane $\pi:S\to \Pb^2$, this time branched over a sextic curve $C\subset\Pb^2$ with two nodes $Q,Q'\in C$. Let $i:S\to S$ be the covering involution and $i^{[2]}:S^{[2]}\to S^{[2]}$ the natural involution.
The invariant lattice of $i$ is generated by the class $c$ of a genus 2 curve which is the pullback of a line, and the classes $d,d'$ of the exceptional divisors obtained by blowing up $Q$ and $Q'$.
Therefore, the invariant lattice of the natural involution is given by
\begin{align*}
H^2(S^{[2]},\Z)^{i^{[2]}}
&=\Z c\oplus\Z d\oplus\Z d'\oplus\Z e\\
&\cong\langle2\rangle\oplus\langle-2\rangle\oplus\langle-2\rangle\oplus\langle-2\rangle,
\end{align*}
where $e$ is half the class of the exceptional divisor on $S^{[2]}$, and $H^2(S,\Z)$ is identified with its image in $H^2(S^{[2]},\Z)$. Hence $(S^{[2]}, i^{[2]})$ is of type $M$, where
\[M:=\varepsilon(\langle2\rangle\oplus\langle-2\rangle\oplus\langle-2\rangle)\oplus\langle-2\rangle\subset L_2\]
and
\[\varepsilon:L_{K3}\hookrightarrow L_2=L_{K3}\oplus\langle-2\rangle\]
is the natural inclusion.
We will implicitly identify $H^2(S^{[2]},\Z)^{i^{[2]}}$ with $M$.
Consider the set
\[\Delta:=\{d,\,d',\,e,\,c-d-d',\,c-d-e,\,c-d'-e\}\]
of  $-2$-classes. The polyhedron 
\[P:=\{x\in \calC_M: (\dd,x)\geq0\text{ for every }\dd\in\Delta\}/\R_{>0}\] is the convex hull of 
\[p_0=c,\ p_1=c-d,\ p'_1=c-d',\ p_2=c-e,\ p_3=2c-d-d'-e.\]
The only non-trivial isometry $\sigma\in\Gamma_M$ preserving $P$ is the involution given by $d\mapsto d'$. Indeed, such an isometry acts on the set of $p_i$, which are uniquely determined as primitive integral representatives of the vertices of $P$. We have 
$(p_0,p_0)=(p_3,p_3)=2$ and 
$(p_1,p_1)=(p'_1,p'_1)=(p_2,p_2)=0$.
Since $\sigma$ extends to $L_2$, we have $\sigma(e+2M)=e+2M$, which shows the claim.

For a $-2$-class $\dd\in M$ let $r_\dd\in\widetilde{O}^+(M)$ be the reflection in the hyperplane $\dd^\perp$ defined by $r_\dd(x)=x+(x,\dd)\dd$. 
Since there is no stable isometry preserving $P$, there is no $\dd\in M$ with $(\dd,\dd)=-2$ such that $\dd^\perp$ meets the interior of $P$. 
By \cite[Thm. 1.2]{vinberg}, the polyhedron $P$ is a fundamental domain for the action of the reflection group 
$\langle r_\dd:\dd\in\Delta\rangle\subset\widetilde{O}^+(M)\subset\Gamma_M$
on $\calC_M/\R_{>0}$.

In order to determine $\kt(M)/\Gamma_M$ it therefore suffices to 
consider elements $\dd\in M$ with $(\dd,\dd)=-10$ and $(\dd,L_2)=2\Z$ such that that the hyperplane $\dd^\perp$ meets the interior of $P$. We can assume that $(\dd,c)\geq 0$ and that there exists a $p\in\{p_1,p'_1,p_2,p_3\}$ with $(\dd,p)<0$. A simple calculation shows that 
\[\dd\in\{\pm(2d-e),\ \pm(2d'-e),\ 2c-3e,\ 2c-2d-2d'-e\}.\]
The corresponding hyperplanes divide $P$ into six polyhedra with vertices
\begin{align*}
P_1&=\{p_0,q_1,q_1',q_2,q_3\},&
P_2&=\{p_2,q_1,q_1',q_2,q_3\},&
P_3&=\{p_0,p_1,q_1,q_2\},\\
P_3'&=\{p_0,p'_1,q_1',q_2\},&
P_4&=\{p_0,p_1,p'_1,q_2\},&
P_5&=\{p_1,p'_1,p_3,q_2\},
\end{align*}
where 
\[q_1=3c-d-2e,\ q_1'=3c-d'-2e,\ q_2=3c-d-d'-2e,\ q_3=3c-2e.\]
We call two such polyhedra \emph{adjacent}, if they have a common face $\dd^\perp$ for some $-10$-class $\dd$.

\begin{figure}[h]
\begin{center}
\begin{tikzpicture}[scale=.7, baseline=-3pt]
\filldraw (0,0) circle (2pt) -- (-2,0) circle (2pt);
\filldraw (0,0) -- (2,1) circle (2pt) -- (4,0) circle (2pt) -- (6,0) circle (2pt);
\filldraw (0,0) -- (2,-1)  circle (2pt)-- (4,0);
\draw (-2,.5) node{$P_2$};
\draw (0,.5) node{$P_1$};
\draw (2,1.5) node{$P_3$};
\draw (2,-1.5) node{$P'_3$};
\draw (4,.5) node{$P_4$};
\draw (6,.5) node{$P_5$};
\end{tikzpicture}
\caption{Adjacency of Kähler-type chambers}
\end{center}
\end{figure}
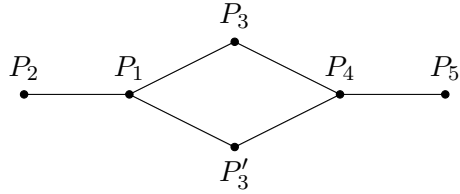

The involution $\sigma$ maps $P_3$ to $P_3'$ and fixes the other $P_i$. Hence the moduli space of simple pairs of type $M$ consists of four components which are Zariski-open subsets
\[\calM_{M,P_i}^0\subset\Omega^+_\Mp/\gmp^+,\ i\neq 3\]
 and one component 
\[\calM_{M,P_3}^0\subset\Omega^+_\Mp/\so^+(\Mp).\] 
Moreover, $\so^+(\Mp)\subset\gmp^+$ is an index 2 subgroup and the projection map
\[\Omega^+_\Mp/\so^+(\Mp)\to\Omega^+_\Mp/\gmp^+\]
is a (branched) double cover. 

We now want to interpret this double cover geometrically.  By \cite[Lemma 13.3]{bayermacri}, the invariant Kähler cone of $S^{[2]}$ corresponds to the polyhedron $P_1$, which is adjacent to $P_2,P_3$ and $P_3'$. 
On the other hand, the fixed locus of $i^{[2]}$ contains three planes:
\begin{enumerate}
\item the symmetric products $D^{(2)},(D')^{(2)}\cong(\Pb^1)^{(2)}\cong\Pb^{2}$, where $D,D'\subset S$ are the exceptional divisors of the blow-ups of $Q,Q'$,
\item the closure of $\{[s,i(s)]\in S^{[2]}:s\in S\setminus S^i\}$, which is isomorphic to $S/i\cong\Pb^2$.
\end{enumerate}

The fact that the planes are contained in the fixed locus implies that the induced involutions on the corresponding flops $X,X',Y$ are biregular (see \cite[Cor. 2.11]{ohashiwandel}).
Since a flop corresponds to a reflection in a $-10$-wall \cite[Rem. 5.2]{mongardiwandel}, the invariant Kähler cones of the flops $X$ and $X'$ correspond to $P_3$ and $P_3'$, and that of $Y$ to $P_2$.

As in Example \ref{exm:nonsep} one can show that $(X,j)$ and $(X',j')$ are not isomorphic for the generic choice of the curve $C\subset\Pb^2$. On the other hand, a deformation of $C\subset\Pb^2$ into itself, such that the two nodes remain nodes and $Q$ and $Q'$ are exchanged, induces a deformation of the flops which deforms $(X,j)$ into $(X',j')$.
\end{exm}

\medskip
\subsection*{Acknowledgment}
I am grateful to my advisor Klaus Hulek for many helpful discussions.

\bibliographystyle{alpha}

\end{document}